\documentclass[a4paper]{article}

\usepackage{amsmath}
\usepackage{float}
\usepackage{graphicx}
\usepackage{subcaption}
\usepackage[colorlinks=true, allcolors=blue]{hyperref}
\usepackage{cite}
\usepackage{tikz}
\usepackage{tikzscale}
\newlength\figureheight
\newlength\figurewidth
\usepackage{pgfplots}
\usepackage{amssymb}
\usetikzlibrary{matrix}
\usepackage{fancyvrb}
\usepackage{pdfpages}
\usepackage{multirow}

\usepackage{amsthm}
\newtheorem{theorem}{Theorem}[section]
\newtheorem{lemma}[theorem]{Lemma}

\theoremstyle{definition}
\newtheorem{definition}{Definition}[section]
\newtheorem{problem}{Problem}[section]

\title{Inverse Eigenvalue Problem and rational Krylov subspaces}
\providecommand{\keywords}[1]{\textit{Keywords: } #1}
\date{}
\author{ Niel Van Buggenhout\footnotemark[2]
	\and Marc Van Barel\footnotemark[2]
	\and Raf Vandebril\footnotemark[2]}
\begin{document}
	\maketitle

	\renewcommand{\thefootnote}{\fnsymbol{footnote}}
	
	\footnotetext[2]{Department of Computer Science, KU Leuven, University of Leuven, 3001 Leuven, Belgium. (niel.vanbuggenhout@kuleuven.be, marc.vanbarel@kuleuven.be, raf.vandebril@kuleuven.be)}
	\footnotetext{The research of the first author was funded by
		the Research Council KU Leuven,
		C1-project C14/17/073 (Numerical Linear Algebra and Polynomial Computations), project
		C14/16/056 (Inverse-free Rational Krylov Methods: Theory and Applications),
		the research of the second author by
		the Research Council KU Leuven, C1-project C14/17/073 (Numerical Linear Algebra and Polynomial Computations), by the Fund for Scientific Research–Flanders (Belgium), EOS Project no 30468160
		and the research of the third author by
		the Research Council KU Leuven, project
		C14/16/056 (Inverse-free Rational Krylov Methods: Theory and Applications
	)}
	
	\begin{abstract}
		The problem of computing recurrence coefficients of sequences of rational functions orthogonal with respect to a discrete inner product is formulated as an inverse eigenvalue problem for a pencil of Hessenberg matrices.
		Two procedures are proposed to solve this inverse eigenvalue problem, via the rational Arnoldi iteration and via an updating procedure using unitary similarity transformations.
		The latter is shown to be numerically stable.
		This problem and both procedures are generalized by considering biorthogonal rational functions with respect to a bilinear form.
		This leads to an inverse eigenvalue problem for a pencil of tridiagonal matrices.
		A tridiagonal pencil implies short recurrence relations for the biorthogonal rational functions, which is more efficient than the orthogonal case.
		However the procedures solving this problem must rely on nonunitary operations and might not be numerically stable.
	\end{abstract}
		
\section{Introduction}\label{intro}
Rational functions are an important component for many numerical methods, e.g., in partial realization problems \cite{GrLi83} and model order reduction \cite{GaGrVD94,GaGrVD96}.
Orthogonal rational functions have attractive properties and are the solution to a least squares rational approximation problem \cite{BuVBVG04}.
This connection has been exploited for the scalar polynomial case \cite{ElGoKa91,BuVB95,Re91,ReAmGr91b} and the vector polynomial case \cite{BuVB95}.
The vector polynomial case can be interpreted as a rational least squares approximation where the poles of the rational functions are not given a priori.
The focus of this manuscript is on rational functions with prescribed poles.\\
Rational functions orthogonal with respect to a discrete inner product or, more general, a discrete bilinear form, are related to structured matrices. While Stieltjes-like or Lanczos-type procedures to construct orthogonal rational functions can suffer from numerical instability, procedures based on structured matrices are numerically stable without sacrificing efficiency \cite{Re91,ReAmGr91b}.\\
The methods proposed here are based on solving an inverse eigenvalue problem (IEP) \cite{Ch98,BoGo87}. For such problems, a matrix of a particular structure is constructed such that it has given spectral properties, a prescribed set of eigenvalues and first entries of normalized eigenvectors.
The IEPs presented and solved in this manuscript generalize the Jacobi matrix IEP \cite{GrHa84} which relates to orthogonal polynomials on the real line and the unitary Hessenberg IEP \cite{ReAmGr91b}, which relates to Szeg\H{o} polynomials \cite{Sz75}.
For orthogonal rational functions an IEP for semiseparable-plus-diagonal matrices \cite{BuVBVG04,VBFaGeMa05} can be formulated. 
Our approach uses a matrix pencil instead, which is a more intuitive representation and provides more flexibility. This flexibility allows the development of a procedure to solve the orthogonal problem that is numerically stable.\\
The former examples are orthogonal with respect to an inner product. 
We also discuss orthogonality with respect to bilinear forms. 
The tridiagonal pencil IEP presented here generalizes polynomials orthogonal to a bilinear form \cite{Re93} and the pseudo-Jacobi inverse eigenvalue problem \cite{XuBeCh19} which arises in the study of quantum mechanics. 
The IEP implies short recurrence relations for rational functions orthogonal with respect to this bilinear form \cite{BeDeZh10}, however the solution procedure must rely on nonunitary transformations and is therefore liable to numerical instability.
\\
Two procedures are proposed to solve the IEPs, one based on Krylov subspace methods and the other on an updating procedure.
Updating procedures start from an available solution of an IEP and efficiently compute the solution to an IEP where the discrete bilinear form or inner product underlying the orthogonal rational functions is enlarged by one node.
The origin of the proposed procedure can be traced back to continued fractions, where Rutishauser \cite{Ru63} used a similar approach to multiply J-fractions.
The core idea of his approach has been used to develop procedures for several inverse eigenvalue problems \cite{GrHa84,Re91,ReAmGr91b,BuVB95}.
Updating an inner product is closely related to a low rank modification of the Cholesky  decomposition of a positive-definite Hermitian matrix \cite{GiGoMuSa74}, or of the LR-factorization of a Hermitian matrix \cite{Be65} in case of a bilinear form.\\
Section \ref{sec:intro} introduces the necessary notions to formally define the main problem of generating sequences of orthogonal rational functions and provides two alternative ways to formulate this problem.
It is formulated as a structured inverse eigenvalue problem and in terms of the factorization of a moment matrix. 
The formulation as an inverse eigenvalue problem relies on the connection between orthogonal rational functions and structured matrices.
Section \ref{sec:sol_Krylov} uses the relation of structured matrices to rational Krylov subspace methods, to provide iterations to solve the main problem.
These iterations are the rational Arnoldi iteration \cite{Ru94} and the rational Lanczos iteration \cite{VBVBVa19}.
Solution procedures based on updating the IEP are introduced in Section \ref{sec:sol_update} which, in the orthogonal case, are numerically stable.
The generalization to the biorthogonal case is new and has the advantage of leading to short recurrence relations for the rational functions, however the procedure is no longer numerically stable.
Numerical experiments in Section \ref{sec:numerics} show that the proposed solution procedures are valid and show the numerical stability of the updating procedure in the orthogonal case.

\section{Problem definition and reformulation}\label{sec:intro}
The main problem considered is the generation of a sequence(s) of (bi)orthogonal rational functions with prescribed poles.
In order to discuss orthogonality, suitable spaces of rational functions, an inner product and a bilinear form must be defined.
Let $\mathcal{P}$ denote the space of polynomials and $\mathcal{P}_n$ the space of polynomials up to degree $n$. Then we define $\mathcal{R}$ to be the space of rational functions of the form
\begin{equation}
r(z) = \frac{p(z)}{q(z)}, \quad p(z), q(z)\in\mathcal{P}.
\end{equation}
A general definition of a bilinear form is given in Definition \ref{def:bilinearForm} and the specific form used here is given by Equation \eqref{eq:bilinear_form}.
\begin{definition}[Bilinear form]\label{def:bilinearForm}
	Let $\mathcal{X}$ and $\mathcal{Y}$ be vector spaces over the field of complex numbers $\mathbb{C}$.
	A bilinear form $\langle .,. \rangle: \mathcal{X} \times \mathcal{Y} \rightarrow \mathbb{C}$ is defined as a mapping which is
	\begin{itemize}
		\item[] linear in the elements of $\mathcal{X}$, i.e., $\langle \alpha_1 x_1 + \alpha_2 x_2, y_1\rangle = \alpha_1 \langle x_1,y_1\rangle + \alpha_2 \langle x_2,y_1\rangle$,
		\item[] linear in the elements of $\mathcal{Y}$, i.e., $\langle x_1, \alpha_1 y_1 + \alpha_2 y_2 \rangle = {\alpha}_1 \langle x_1, y_1 \rangle + {\alpha}_2 \langle x_1, y_2 \rangle$.
	\end{itemize}
\end{definition}
If $x\in \mathcal{X}$ and $y\in \mathcal{Y}$ satisfy $\langle x,y \rangle = 0$, then $x$ and $y$ are said to be orthogonal with respect to the bilinear form.
Set $\mathcal{X}=\mathcal{Y}=\mathcal{R}$ and consider weights $\{v_i\}_{i=1}^m$, $\{w_i \}_{i=1}^m$, with $v_i,w_i \in \mathbb{C}\backslash \{0\}$, and distinct nodes $\{z_i\}_{i=1}^m$, $z_i\in \mathbb{C}$, then the discrete bilinear form for $r,s\in \mathcal{R}$ under consideration is
\begin{equation}\label{eq:bilinear_form}
\langle r,s \rangle := \sum_{i=1}^{m} \bar{w}_i v_i {s}(z_i) r(z_i).
\end{equation}
The poles of the rational functions $r,s$ must be different from the nodes $z_i$.

Orthogonality with respect to a bilinear form leads to biorthogonal rational functions. 
Biorthogonality reduces to orthogonality when the bilinear form reduces to an inner product.
Definition \ref{def:innerProd} provides a formal definition of an inner product.
\begin{definition}[Inner product]\label{def:innerProd}
	Let $\mathcal{X}$ be a vector space over the field of complex numbers $\mathbb{C}$. 
	A map $(.,.): \mathcal{X}\times \mathcal{X} \rightarrow \mathbb{C}$ is called an \textit{inner product} if it is sesquilinear form, i.e., linear in its first and anti-linear in its second variable:
	\begin{itemize}
		\item[] $( \lambda x + \mu y,z ) = \lambda (x,z) + \mu (y,z) $,
		\item[] $(x, \lambda y + \mu z)  = \bar{\lambda} (x,y) + \bar{\mu} (x,z) $ for all $x,y,z\in \mathcal{X},\quad \lambda,\mu \in \mathbb{C}$.
	\end{itemize}
	Furthermore, it must be Hermitian
	\begin{equation*}
	(y,x) = \overline{(x,y)},
	\end{equation*}
	and positive definite
	\begin{equation*}
	(x,x) >0 \text{ for all }x\in \mathcal{X}, \quad x\neq0.
	\end{equation*}
\end{definition}
For discrete inner products the positive definiteness implies that the associated vector space must be finite dimensional.
A suitable $(n+1)$-dimensional vector space is $\mathcal{R}_n^\Xi \subset \mathcal{R}$, where $\Xi = \{\xi_i \}_{i=1}^{n}$, $\xi_i\in \bar{\mathbb{C}}:=\mathbb{C} \cup \{\infty \}$, is the set of prescribed poles.
A rational function $r \in \mathcal{R}_n^\Xi$ is defined as
\begin{equation}\label{eq:rationalFunctionPoles}
r(z) = \frac{p_n(z)}{q(z;\Xi)}, \quad q(z;\Xi) := \underset{\xi_i \neq \infty }{\prod_{i=1}^{n}}(z-\xi_i), \quad \deg(p_n(z)) \leq n.
\end{equation}
If $\Xi= \{\xi_i \}_{i=1}^{k}$, with $k>n$, then the first $n$ poles $\xi_i$, $i=1,2,\dots,n$, are used in the construction of the space $\mathcal{R}_n^\Xi$.
The discrete inner product for $r,s\in \mathcal{R}_n^{\Xi}$, with $n<m$, real weights $\{\alpha_i\}_{i=1}^m$, $\alpha_i \in \mathbb{R}\backslash \{0\}$, and distinct nodes $\{z_i\}_{i=1}^m$, $z_i\in \mathbb{C}$, is
\begin{equation}\label{eq:innerProduct}
(r,s) := \sum_{i=1}^{m} \alpha_ i \overline{s(z_i)} r(z_i).
\end{equation}
Again, the poles of $r,s$ must be distinct from the nodes and $\overline{s(z)} := \bar{s}(\bar{z})$. For polynomials $p(z) = \sum_{i=0}^{n}c_i z^i$, let $\bar{p}(z) := \sum_{i=0}^{n}\bar{c}_i z^i$ and, for rational functions $s(z) = \frac{p(z)}{q(z)}$, let $\bar{s}(z) := \frac{\bar{p}(z)}{\bar{q}(z)}$.
On $\mathcal{R}_k^\Xi$ with $k \geq m$ expression \eqref{eq:innerProduct} would be a Hermitian sequilinear form.\\
The spaces $\mathcal{R}$ and $\mathcal{R}_n^{\Xi}$ together with the forms $\langle .,. \rangle$ and $(.,.)$ allow a formal definition of the main problem.
This problem is the computation of nested (bi)orthogonal bases for nested subspaces of rational functions with prescribed poles. 
The problem is formulated more precisely in Problem \ref{def:functionalProblem_orthog} for orthogonal and in Problem \ref{def:functionalProblem_biorthog} for biorthogonal rational functions.
\begin{problem}[Generating orthogonal rational functions]\label{def:functionalProblem_orthog}
	For $n<m$, consider a discrete inner product \eqref{eq:innerProduct} and poles $\Xi = \{\xi_i \}_{i=1}^{n}$, $\xi_i \in \bar{\mathbb{C}}$. Construct a sequence of orthonormal rational functions $\{r_k\}_{k=0}^{n}$, $r_k\in \mathcal{R}^{\Xi}_k \backslash \mathcal{R}^{\Xi}_{k-1}$, orthonormal with respect to the given inner product, i.e.,
	\begin{equation*}
	( r_k,r_l) = \sum_{i=1}^m \alpha_i \overline{r_l(z_i)} r_k(z_i) = \delta_{k,l}.
	\end{equation*}
\end{problem}

\begin{problem}[Generating biorthogonal rational functions]\label{def:functionalProblem_biorthog}
	For $n<m$, consider a discrete bilinear form \eqref{eq:bilinear_form} and poles $\Xi = \{\xi_i \}_{i=1}^{n}$ and $\varPsi = \{\psi_i \}_{i=1}^{n}$, $\xi_i,\psi_i \in \bar{\mathbb{C}}$. Construct two sequences of rational functions $\{r_k\}_{k=0}^{n}$, $r_k\in \mathcal{R}^{\Xi}_k \backslash \mathcal{R}^{\Xi}_{k-1}$ and $\{s_k\}_{k=0}^{n}$, $s_k\in \mathcal{R}^{\varPsi}_k\backslash \mathcal{R}^{\varPsi}_{k-1}$ such that they are orthonormal to each other with respect to the given bilinear form, i.e., 
	\begin{equation*}
	\langle r_k,s_l\rangle = \sum_{i=1}^m \bar{w_i}v_i {s}_l(z_i) r_k(z_i) = \delta_{k,l}.
	\end{equation*}
\end{problem}

These problems can be formulated as problems in matrix theory, namely, inverse eigenvalue problems.
Section \ref{sec:relation} introduces rational Krylov subspaces, which are spaces spanned by matrix vector products. The associated inner product is the Euclidean inner product.
The relationship of such spaces together with the Euclidean inner product to the spaces and bilinear form discussed above allows the formulation of two structured inverse eigenvalue problems in Section~\ref{sec:IEP}.
Another possible formulation is given in Section \ref{sec:moment}, which discusses the problem of factoring a moment matrix.
This formulation serves as a theoretical tool to use in a uniqueness argument.

\subsection{Rational Krylov subspaces}\label{sec:relation}
Rational Krylov subspaces \cite{Ru84} are used to formulate Problem \ref{def:functionalProblem_orthog} and \ref{def:functionalProblem_biorthog} as problems in matrix theory. These subspaces are defined in Definition \ref{def:rks}.
In this definition a pole $\xi$ is denoted as a fraction $\frac{\nu}{\mu}$ to allow for a unified representation of $\xi\in \mathbb{C}$ and $\xi = \infty$.
To keep the notation concise, we will usually write the poles as $\xi$ and whenever relevant we will split them up in its fraction $\frac{\nu}{\mu}$. 
For any square matrices $B,C$ the notation $\frac{B}{C} := BC^{-1} = C^{-1}B$ and is only valid when $B$ and $C^{-1}$ commute, which is the case for matrices used in the rational Krylov subspace.
\begin{definition}[Rational Krylov subspace (RKS)]\label{def:rks}
	Consider $A \in \mathbb{C}^{m\times m}$, $v\in \mathbb{C}^m$ and $\Xi = \{\xi_i \}_{i=1}^{n-1}$, then a rational Krylov subspace with poles $\xi_i = \frac{\nu_i}{\mu_i}\in \bar{\mathbb{C}}$ and shift parameters $\rho_i, \eta_i\in\mathbb{C}$ is defined as
	\begin{equation*}
	\mathcal{K}_n(A,v;\Xi):=  \textrm{span}\left\{v, \frac{\rho_1 A-\eta_1 I}{\mu_1 A-\nu_1 I} t_1,\dots , \frac{\rho_{n-1} A-\eta_{n-1} I}{\mu_{n-1} A-\nu_{n-1} I} t_{n-1} \right\}.
	\end{equation*}
\end{definition}	
\noindent The parameters $\rho_i, \eta_i, \mu_i, \nu_i$ must satisfy $\frac{\eta_i}{\rho_i} \neq \frac{\nu_i}{\mu_i}$ and the continuation vectors $t_i\in\mathcal{K}_i(A,v;\Xi)$, $i = 1,2,\dots, n-1$, must be chosen such that $\dim(\mathcal{K}_n(A,v;\Xi)) = n$ for all $n \leq m$.
In general this statement should be more precise, since one can find invariant subspaces, i.e., $\dim(\mathcal{K}_n(A,v;\Xi)) < n\leq m$, e.g., the minimal polynomial \cite[p.19]{LiSt13}. 
However in this manuscript such scenarios cannot occur because of conditions that will be imposed on $A$ and $v$. \\
Vectors $k_i$ spanning a RKS, $\mathcal{K}_n(A,v;\Xi) = \textrm{span}\{k_1, k_2,\dots, k_{n}\},$ can be written in terms of rational functions  $r_i^\Xi\in  \mathcal{R}^{\Xi}_i$, i.e., $k_i = r_{i-1}^\Xi (A)v$.
The inner product for vectors in rational Krylov subspaces $\mathcal{K}_n(A,v;\Xi)\subseteq \mathbb{C}^m$ is the usual inner product on $\mathbb{C}^m$, i.e., for $x,y\in \mathbb{C}^m$,
\begin{equation}\label{eq:vecInnerProduct}
\langle x,y \rangle := y^H x.
\end{equation}
Our interest goes out to diagonal matrices $A=\textrm{diag}(\{z_i\}_{i=1}^m)$, where $z_i$ are the nodes of some discrete bilinear form or inner product for rational functions.
Lemma~\ref{lemma:bilinearForms_orthog} provides an equivalence relation between inner product \eqref{eq:vecInnerProduct} and inner product \eqref{eq:innerProduct} and Lemma \ref{lemma:bilinearForms_biorthog} between inner product \eqref{eq:vecInnerProduct} and bilinear form \eqref{eq:bilinear_form}.
For simplicity these are formulated in terms of diagonal matrices, however, they hold for more general matrices.

\begin{lemma}\label{lemma:bilinearForms_orthog}
	Consider inner product \eqref{eq:innerProduct} on the space $\mathcal{R}_{n-1}^\Xi$, with nodes $\{z_i\}_{i=1}^m \in \mathbb{C}$, $z_i\neq z_j$ for $i\neq j$, and weights $\{\alpha_i\}_{i=1}^{m}\in \mathbb{R}\backslash \{0\}$. Let $Z = \textrm{diag}\{z_1,\dots,z_m\}$ and $v=\begin{bmatrix}
	v_1 & \dots & v_m
	\end{bmatrix}^\top$, $v_i\in \mathbb{C}\backslash \{0\}$.
	Then, this inner product $( .,. ): \mathcal{R}_{n-1}^\Xi\times \mathcal{R}_{n-1}^\Xi\rightarrow \mathbb{C}$, with $\alpha_i=\vert v_i\vert^2$, and, the inner product \eqref{eq:vecInnerProduct} restricted to rational Krylov subspaces, $\langle .,. \rangle: \mathcal{K}_n(Z,v;\Xi)\times \mathcal{K}_n(Z,v;\Xi)\rightarrow \mathbb{C}$, are the same.
\end{lemma}
\begin{proof}
	Consider $x,y\in \mathcal{K}_n(Z,v;\Xi)$, then there exist $r,s \in \mathcal{R}^\Xi_{n-1}$ such that $x = r(Z)v$ and $y = s(Z)v$.
	Apply \eqref{eq:vecInnerProduct} for $x,y$ and the statement follows for $\alpha_i = \vert v_i\vert ^2$,
	\begin{equation*}
	\langle r(Z)v,s(Z)v \rangle  = v^H\bar{s}(Z^H)r(Z)v 
	= \sum_{i=1}^m \bar{v}_i v_i \overline{s(z_i)}r(z_i) = (r,s).
	\end{equation*}\qed
\end{proof}

\begin{lemma}\label{lemma:bilinearForms_biorthog}
	Consider bilinear form \eqref{eq:bilinear_form} on the space $\mathcal{R}$, with nodes $\{z_i\}_{i=1}^m \in \mathbb{C}$, $z_i\neq z_j$ for $i\neq j$, and weights $\{v_i\}_{i=1}^{m}$, $\{w_i\}_{i=1}^{m}$, $v_i,w_i \in \mathbb{C}\backslash \{0\}$. Let $Z = \textrm{diag}\{z_1,\dots,z_m\}$ and $v=\begin{bmatrix}
	v_1 & \dots & v_m
	\end{bmatrix}^\top$, $w=\begin{bmatrix}
	w_1 & \dots & w_m
	\end{bmatrix}^\top$.
	Then, the bilinear form \eqref{eq:bilinear_form} restricted to rational functions with prescribed sets of poles $\Xi,\varPsi$, i.e., $\langle .,. \rangle: \mathcal{R}_{n-1}^\Xi\times \mathcal{R}_{n-1}^\varPsi\rightarrow \mathbb{C}$ and, the inner product \eqref{eq:vecInnerProduct} restricted to rational Krylov subspaces, $\langle .,. \rangle: \mathcal{K}_n(Z,v;\Xi)\times \mathcal{K}_n(Z^H,w;\varPsi)\rightarrow \mathbb{C}$, are the same.
\end{lemma}
\begin{proof}
	Consider $x\in \mathcal{K}_n(Z,v;\Xi)$ and $y\in \mathcal{K}_n(Z^H,w;\varPsi)$, then there exist $r \in \mathcal{R}^\Xi_{n-1}$ and $s \in \mathcal{R}^\varPsi_{n-1}$ such that $x = r(Z)v$ and $y = s(Z^H)w$.
	Apply \eqref{eq:vecInnerProduct} for $x,y$ and the statement follows
	\begin{equation*}
	\langle r(Z)v,s(Z^H)w \rangle  = w^H\bar{s}(Z)r(Z)v 
	= \sum_{i=1}^m \bar{w}_i v_i \bar{s}(z_i)r(z_i) = \langle r,\bar{s}\rangle.
	\end{equation*}\qed
\end{proof}
The consequence of Lemma \ref{lemma:bilinearForms_orthog} (and Lemma \ref{lemma:bilinearForms_biorthog}) is that the recurrence coefficients that generate orthogonal (biorthogonal) vectors in some RKS can be directly used to construct orthogonal (biorthogonal) rational functions. These functions will be orthogonal with respect to a discrete inner product (bilinear form) that is composed of the eigenvalues of $Z$ and elements of $v$, which make up the RKS.
And exactly this connection is exploited by inverse eigenvalue problems.
Remark that the restriction of inner product \eqref{eq:vecInnerProduct} in Lemma \ref{lemma:bilinearForms_biorthog} is no longer an inner product on the restricted domain.

\subsection{Inverse eigenvalue problems}\label{sec:IEP}
Problem \ref{def:functionalProblem_orthog} and \ref{def:functionalProblem_biorthog} can be formulated as a problem in matrix theory using Lemma~\ref{lemma:bilinearForms_orthog} and \ref{lemma:bilinearForms_biorthog}, respectively. They take the form of inverse eigenvalue problems.
The considered structured inverse eigenvalue problem is provided in Problem \ref{def:SIEP}.
\begin{problem}[Structured IEP]\label{def:SIEP}
	Given nodes $\{z_1,\dots,z_m \}$, $z_i\in \mathbb{C}$, find a matrix $B\in \mathcal{N}$ such that the spectrum of $B$ equals the given nodes, i.e., $\sigma(B) = \{z_1,\dots,z_m  \}$. 
\end{problem}
Throughout this section, $\mathcal{N} = \{N_i \}_i$ denotes the set of matrices with a particular structure. Structure can be imposed by sparsity conditions, i.e., certain elements of $N_i$ must be zero, or rank conditions on submatrices of $N_i$. 
Examples of the former are Hessenberg matrices, tridiagonal matrices and of the latter are semiseperable matrices, which have lower-and upper triangular parts of rank one. 
Additional conditions such as symmetry or positive definiteness can also be imposed, e.g., a Jacobi matrix \cite{GrHa84}.
Besides the spectrum, additional spectral data is usually provided in Problem \ref{def:SIEP}.
In this manuscript the additional spectral data provided is information about the eigenvectors. Let $e_i$ denote the $i$th canonical unit vector of appropriate size, Problem \ref{def:SIEP_matrix} formulates the general IEP used throughout this manuscript.
\begin{problem}[Structured IEP-matrix formulation]\label{def:SIEP_matrix}
	Given a diagonal matrix $Z = \textrm{diag}(z_1,\dots, z_m)\in \mathbb{C}^{m\times m}$, $z_i\neq z_j$ for $i\neq j$, and vectors $v,w\in \mathbb{C}^m$. 
	Find matrix $B\in \mathcal{N}$ and $V\in \mathbb{C}^{m\times m}$, where $V e_1 = \frac{v}{\nu}$ and $V^{-H}e_1 = \frac{w}{\eta}$, with $\bar{\eta} \nu = \langle v,w\rangle$, such that
	\begin{equation}\label{eq:similar}
	V^{-1}ZV = B.
	\end{equation}
\end{problem}
Since $V^{-1} Z = B V^{-1}$, $B$ is said to represent $Z$ in the basis for $\mathbb{C}^{m}$ formed by the columns of $V$. Another interpretation, more natural for the equation $ZV = VB$, is that $B$ contains the recurrence coefficients for constructing the columns of $V$.
The connection between the spaces of rational functions and rational Krylov subspaces, as described in Section \ref{sec:relation} allows the formulation of two IEPs which are equivalent to the problem of constructing sequences of (bi)orthogonal rational functions. These IEPs enforce the structure
\begin{itemize}
	\item $\mathcal{N}_1 = \{\text{Hessenberg matrix pencils} \}$,
	\item $\mathcal{N}_2 = \{\text{Tridiagonal matrix pencils} \}$,
\end{itemize}
and correspond, respectively, to Problem \ref{def:functionalProblem_orthog} and \ref{def:functionalProblem_biorthog}.
These equivalent problems are formulated in Problem \ref{def:HPIEP} and \ref{def:TPIEP}.
Denote by $a_{i,j}$ the element at row $i$ and column $j$ of some matrix $A$.
\begin{problem}[Hessenberg pencil IEP (HPIEP)]\label{def:HPIEP}
	Consider a diagonal matrix $Z =\textrm{diag}(z_1,\dots, z_m)$ of distinct nodes $z_i\in \mathbb{C}$, a vector of weights $v\in \mathbb{C}^m$ and a set of poles $\Xi = \{\xi_i\}_{i=1}^{m-1}$, $\xi_i \in \overline{\mathbb{C}}$. Construct a Hessenberg pencil $(H,K)\in\mathbb{C}^{m\times m}$, with $\frac{h_{i+1,i}}{k_{i+1,i}} = \xi_i$, $i=1,2,\dots, m-1$, and a unitary matrix $Q\in \mathbb{C}^{m\times m}$ such that
	\begin{equation}
	Q^H Z Q K = H \quad \text{and}\quad Q e_1  = \frac{v}{\Vert v \Vert}.
	\end{equation}
\end{problem}
This problem implies, by the rational implicit Q theorem \cite[Theorem 5.1]{CaMeVa19_QZ}, that the column span of $Q$ equals $\mathcal{K}_{m}(Z,v;\Xi)$. In this theorem the key is to note the correspondence between the poles in the Krylov subspace and the ratio of the subdiagonal elements of the Hessenberg pencil. 
Solving this IEP corresponds to computing recurrence coefficients of a sequence of orthogonal rational functions, orthogonal with respect to a discrete inner product on $\mathcal{R}_{m-1}^\Xi$. 
These rational functions form a nested basis for $\mathcal{R}^\Xi_{m-1}$. Hence, they form a solution to Problem \ref{def:functionalProblem_orthog}.

\begin{problem}[Tridiagonal pencil IEP (TPIEP)]\label{def:TPIEP}
	Consider a diagonal matrix $Z =\textrm{diag}(z_1,\dots, z_m)$ of distinct nodes $z_i\in \mathbb{C}$, two vectors containing weights $v,w\in \mathbb{C}^m$, $\langle v,w \rangle \neq 0$, and sets of poles $\Xi = \{\xi_i\}_{i=1}^{m-1}$ and $\varPsi = \{\psi_i \}_{i=1}^{m-2}$, $\xi_i, \psi_i \in \overline{\mathbb{C}}$. Construct a tridiagonal pencil $(T,S)\in\mathbb{C}^{m\times m}$, with $\frac{t_{i,i+1}}{s_{i,i+1}} = \xi_i$, $i=1,2,\dots, m-1$ and $\frac{t_{i+1,i}}{s_{i+1,i}} = \bar{\psi}_{i-1}$, $i=2,3,\dots, m-1$ and matrices $V,W\in \mathbb{C}^{m\times m}$ such that
	\begin{equation}
	W^H Z V S = T \quad \text{and}\quad V e_1 = \frac{v}{\nu}, \quad We_1 = \frac{w}{\eta},
	\end{equation}
	where $W^H V = I$ (biorthogonality) and $w^H v = \nu \bar{\eta}$.
\end{problem}
Thus a tridiagonal pencil $(T,S)$, where both the subdiagonal and superdiagonal elements should satisfy some restrictions on their ratios, must be constructed.
These ratio restrictions guarantee that the corresponding rational Krylov subspaces have the appropriate poles \cite{VBVBVa19}. 
Solving this IEP corresponds to computing recurrence coefficients of one of two sequences of biorthogonal rational functions, orthogonal with respect to a discrete bilinear form.
Hence, it means solving Problem \ref{def:functionalProblem_biorthog}. 
Note that the tridiagonal pencil $(T,S)$ only provides recurrence coefficients for one of the sequences of biorthogonal rational functions. 
To obtain the other sequence, an additional tridiagonal pencil $(\widetilde{T},\widetilde{S})$ must be constructed such that $V^H Z^H W \widetilde{S} = \widetilde{T}$.

\subsection{Moment matrix formulation}\label{sec:moment}
The essence of the problems defined above is the construction of nested biorthogonal bases for some given vector spaces. 
This can also be described using moment matrices.
This moment matrix formulation will allow us, by a uniqueness argument, to show that the procedures proposed to solve the IEPs, introduced in Section \ref{sec:sol_Krylov} and \ref{sec:sol_update}, indeed solve the correct problem.
The moment matrix $M\in \mathbb{C}^{n\times n}$, associated to two sets $X = \{x_{i}\}_{i=1}^n, Y=\{y_{i}\}_{i=1}^n$ of vectors in some vector spaces $\mathcal{X},\mathcal{Y}$ and a bilinear form $\langle .,. \rangle: \mathcal{X}\times \mathcal{Y} \rightarrow \mathbb{C}$, is defined as
\begin{equation}\label{eq:momentMatrix_bilinearForm}
M:=\begin{bmatrix}
\langle x_{i},y_{j}  \rangle
\end{bmatrix}_{i,j=1}^n.
\end{equation}
Lemma \ref{theorem:momentLR} states how the sets $X,Y$ can be orthogonalized using a factorization of the moment matrix $M$.
\begin{lemma}[Biorthonormal vectors via moment matrix factorization\cite{Da75}]\label{theorem:momentLR}
	Let $X = \begin{bmatrix}
	x_1 & x_2 & \dots & x_n
	\end{bmatrix}$ and $Y = \begin{bmatrix}
	y_1 & y_2 & \dots & y_n
	\end{bmatrix}$ denote matrices containing linearly independent vectors in some vector spaces $\mathcal{X}$ and $\mathcal{Y}$, respectively. Let $M$ be the associated moment matrix \eqref{eq:momentMatrix_bilinearForm}, which is assumed to be strongly nonsingular.
	Then the (non-pivoted) LR factorization is $M=LR$, with $L$ a lower and $R$ an upper triangular matrix. These factors allow the construction of biorthonormal sets of vectors $\{ v_i\},\{w_i\}$ whose span equals that of $\{x_i\}, \{y_i\}$.
	Set $V = 
	\begin{bmatrix}
	v_1 & v_2 & \dots & v_n
	\end{bmatrix} := X R^{-1}$ and $W  = \begin{bmatrix}
	w_1 & w_2 & \dots & w_n
	\end{bmatrix} := Y L^{-\top}$, then these sets are nested
	\begin{align*}
	\textrm{span}\{v_1,\dots, v_i \} &= \textrm{span}\{x_1,\dots, x_i\}\\
	\textrm{span}\{w_1,\dots, w_i \} &= \textrm{span}\{y_1,\dots, y_i \}, \quad i=1,\dots, n,
	\end{align*}
	and biorthonormal
	\begin{align*}
	\begin{bmatrix}
	\langle v_{i},w_{j}  \rangle
	\end{bmatrix}_{i,j=1}^n &= L^{-1} M R^{-1} = I.
	\end{align*}
\end{lemma}
The nestedness of the sets $V,W$ follows from the upper triangular structure of $R^{-1}$ and $L^{-\top}$. 
Note that the moment matrix can also be constructed for an inner product, then $L^{-\top}$ should be replaced by $L^{-H}$ in the above lemma.
For an inner product and $X=Y$ the moment matrix is Hermitian and positive definite. Hence, the Cholesky decomposition can be used to factor $M$.\\
In this manuscript the spaces considered are spaces of rational functions with prescribed poles or rational Krylov subspaces. 
Moment matrices arising from rational Krylov subspaces can be elegantly represented.\\
Consider $Z=\textrm{diag}(\{z_i\}_{i=1}^m)$, with distinct elements, and rational Krylov subspaces $\mathcal{K}_n(Z,v;\Xi)$, $\mathcal{K}_n(Z^H,w;\Psi)$,
which, respectively, have nested bases
\begin{align}
K_v& = \begin{bmatrix}
t_0(Z)v & t_1(Z)v & \dots &  t_{n-1}(Z)v
\end{bmatrix},\label{eq:KrylovBasisV}\\
K_w &= \begin{bmatrix}
u_0(Z^H)w & u_1(Z^H)w & \dots & u_{n-1}(Z^H)w
\end{bmatrix} \label{eq:KrylovBasisW},
\end{align}
where $t_i\in \mathcal{R}_i^\Xi$ and $u_i\in \mathcal{R}_i^\varPsi$. They will be referred to as \emph{Krylov bases}.
Using \eqref{eq:vecInnerProduct}, we can construct the following two moment matrices. 
The first moment matrix
\begin{equation*}
K_v^H K_v = \begin{bmatrix}
\langle t_i(Z)v,t_j(Z)v \rangle
\end{bmatrix}_{i,j=0}^{n-1}
\end{equation*}
can be interpreted by Lemma \ref{lemma:bilinearForms_orthog} as a moment matrix generated by an inner product of the form \eqref{eq:innerProduct} for a sequence of rational functions $\{t_i(z)\}_i$, $t_i\in \mathcal{R}_i^\Xi$.
The second moment matrix
\begin{equation}\label{eq:momentMatrix}
K_w^H K_v = \begin{bmatrix}
\langle t_i(Z)v,u_j(Z^H)w \rangle
\end{bmatrix}_{i,j=0}^{n-1},
\end{equation}
is related, by Lemma \ref{lemma:bilinearForms_biorthog}, to the moment matrix for bilinear form \eqref{eq:bilinear_form} generated by rational functions $\{t_i(z)\}_i$, $t_i\in \mathcal{R}_i^\Xi$ and $\{u_j(z)\}_j$, $u_j \in \mathcal{R}_j^\varPsi$.\\
Thus essentially we are looking for a numerical procedure to (implicitly) compute the LR factorization of a moment matrix.
The added difficulty in Problem \ref{def:HPIEP} and \ref{def:TPIEP} is that the structure of the matrices containing the recurrence coefficients must adhere to some conditions, i.e., they must be a Hessenberg and tridiagonal pencil, respectively. Two procedures to solve these IEPs are proposed in Section \ref{sec:sol_Krylov} and \ref{sec:sol_update}.

\section{Krylov subspace methods}\label{sec:sol_Krylov}
Based on procedures to compute bases of rational Krylov subspaces, we propose methods to solve the IEPs formulated in Problem \ref{def:HPIEP} and \ref{def:TPIEP}.
A common assumption for such methods is that they do not break down, i.e., they will run to completion and the resulting basis will span $\mathbb{C}^m$ if the Krylov space is generated using an $m\times m$ matrix.
Details about breakdowns can be found in literature \cite{Wi65,Gu97,LiSt13}.
In the remainder of this text we assume that no breakdowns occur, this is an appropriate setting since we have full control over the matrix and the starting vector which generate the Krylov subspaces.\\
The Hessenberg pencil IEP is solved by applying rational Arnoldi iteration \cite{Ru84,Ru94} in Section \ref{sec:RatArnoldi} and in our setting this iteration cannot break down. 
The solution of the tridiagonal pencil IEP is discussed in Section \ref{sec:RatLanczos} and relies on a rational Lanczos iteration \cite{VBVBVa19}.
Due to its biorthogonal nature, this iteration can break down an therefore we must rely on our assumption.

\subsection{Rational Arnoldi iteration}\label{sec:RatArnoldi}
The rational Arnoldi iteration constructs an orthogonal basis $Q$ for the rational Krylov space $\mathcal{K}_m(A,v;\Xi)$, where $A\in \mathbb{C}^{m\times m}$, $v \in \mathbb{C}^m$ and $\Xi = \{\xi_i \}_{i= 1}^{m-1}$, $\xi_i\in\bar{\mathbb{C}}$ are given. 
The idea, following Ruhe \cite{Ru94}, is to construct $Q= \begin{bmatrix}
q_1 & \dots & q_m
\end{bmatrix}$ column by column such that $Q^H Q = I$, $\textrm{span}\{q_1,\dots, q_i \} = \mathcal{K}_i(A,v;\Xi)$ for $i=1,\dots,m$, and $Q^H A QK = H$, with $H$ and $K$ Hessenberg matrices.
The rational Arnoldi iteration is derived below, since it will provide insight in the relation between the Hessenberg pencil and the rational Krylov subspace.\\
The derivation starts by assuming, under induction hypothesis, that $\{q_1, \dots, q_k\}$ has already been computed, such that $\textrm{span}\{q_1, \dots, q_k\} = \mathcal{K}_k(A,v;\Xi)$. 
This space is expanded to $\mathcal{K}_{k+1}(A,v;\Xi)$ by computing a vector $\tilde{q}_{k+1}$ which has a component in $\mathcal{K}_{k+1}(A,v;\Xi)\backslash \mathcal{K}_{k}(A,v;\Xi)$,
\begin{equation}\label{eq:expand}
\tilde{q}_{k+1} = (\rho_k A-\eta_k I)(\mu_k A-\nu_k I)^{-1} t_k, \quad t_k\in \mathcal{K}_k(A,v;\Xi).
\end{equation}
Next the Gram-Schmidt orthogonalization procedure is performed on $\tilde{q}_{k+1}$ to enforce orthogonality to $\mathcal{K}_{k}(A,v;\Xi)$,
\begin{equation}\label{eq:orthog}
h_{k+1,k}q_{k+1} = \tilde{q}_{k+1} - \sum_{i=1}^k \langle \tilde{q}_{k+1}, q_i \rangle q_i.
\end{equation}
Then $\textrm{span}\{q_1,\dots,q_k,q_{k+1} \} = \mathcal{K}_{k+1}(A,v;\Xi)$ and for appropriate normalization $\langle q_{k+1},q_i \rangle = \delta_{i,k+1}$, $i\leq k+1$.
Set $h_{i,k}:= \langle \tilde{q}_{k+1}, q_i \rangle$, for $i=1,\dots, k$, and $h_{k+1,k}$ follows from normalization. 
Combine \eqref{eq:expand} and \eqref{eq:orthog} to obtain the recurrence relation
\begin{equation*}
A \begin{bmatrix}
q_1 & \cdots & q_{k-1} & q_k & q_{k+1}\end{bmatrix}  \begin{bmatrix}
\mu_k h_{1,k}\\
\vdots\\
\mu_k h_{k-1,k}\\
\mu_k h_{k,k} - \rho_k\\
\mu_k h_{k+1,k}
\end{bmatrix} = \begin{bmatrix}
q_1 & \cdots &  q_{k+1}\end{bmatrix}  \begin{bmatrix}
\nu_k h_{1,k}\\
\vdots\\
\nu_k h_{k-1,k}\\
\nu_k h_{k,k}-\eta_k\\
\nu_k h_{k+1,k}
\end{bmatrix}.
\end{equation*}
The recurrence coefficients are collected in the $(k+1)\times k$ Hessenberg matrix pencil $(\widetilde{H}_k,\widetilde{K}_k)$ formed by
\begin{align*}
\widetilde{H}_k &= \begin{bmatrix}
h_{11}-\eta_1/\mu_1 & \dots & h_{1k}\\
h_{21} & \ddots & \vdots\\ 
& \ddots & h_{kk}-\eta_k/\mu_k\\
&  & h_{k+1,k}
\end{bmatrix} \textrm{diag}(\mu_1,\dots,\mu_k),\\
\widetilde{K}_k &= \begin{bmatrix}
h_{11}-\rho_1/\nu_1 & \dots & h_{1k}\\
h_{21} & \ddots & \vdots\\ 
& \ddots & h_{kk}-\rho_k/\nu_k\\
&  & h_{k+1,k}
\end{bmatrix} \textrm{diag}(\nu_1,\dots,\nu_k).
\end{align*}
The poles of the RKS appear as ratios on the subdiagonal of this pencil, $\frac{\tilde{h}_{i+1,i}}{\tilde{k}_{i+1,i}} = \frac{\mu_i}{\nu_i} = \xi_i$.
For $k=m$ it is easy to see that $h_{m+1,m} = 0$ and thus $AQ_m K = Q_m H$ is satisfied, where $H,K\in \mathbb{C}^{m\times m}$ are obtained by truncating the $(m+1)$th row of $\widetilde{H}_m,\widetilde{K}_m$.
And $Q:=Q_m$ forms an orthonormal basis for $\mathcal{K}_m(A,v;\Xi)$.
Theorem~\ref{theorem:sol_Arnoldi} states how this procedure can be used to solve Problem~\ref{def:HPIEP}.
\begin{theorem}\label{theorem:sol_Arnoldi}
	Let the unitary matrix $Q\in \mathbb{C}^{m\times m}$ and Hessenberg matrices $H,K\in \mathbb{C}^{m\times m}$ be obtained by applying the rational Arnoldi iteration for $\mathcal{K}(Z,v;\Xi)$, with $Z$, $v$ and $\Xi$ as in Problem \ref{def:HPIEP}.
	Then these matrices solve Problem \ref{def:HPIEP}.
\end{theorem}
\begin{proof}
	Since $\mathcal{K}_m(Z,v;\Xi) = \mathbb{C}^m$ by the conditions on $Z$ and $v$, $Q$ is unitary. Thus $Q^{-1}ZQK=H$, i.e., the spectrum of $(H,K)$ equals $\sigma(Z)$. The remainder of the proof follows immediately from the properties of the matrices obtained by the rational Arnoldi iteration.\qed
\end{proof}

\subsection{Rational Lanczos Iteration}\label{sec:RatLanczos}
The rational Lanczos iteration \cite{VBVBVa19} constructs a pair of biorthogonal bases $V,W$ for rational Krylov subspaces $\mathcal{K}_m(A,v;\Xi)$ and $\mathcal{K}_m(A^H,w;\Psi)$, respectively. Here $A\in \mathbb{C}^{m\times m}$, $v,w \in \mathbb{C}^m$, $\langle v,w \rangle \neq 0$, and $\Xi = \{\xi_i \}_{i= 1}^{m-1}$, $\Psi = \{\psi_i \}_{i= 1}^{m-1}$, $\xi_i, \psi_i\in \bar{\mathbb{C}}$. The recurrence relation underlying this iteration is
$AV_{k+1} S_k = V_{k+1} T_k,$
where the recurrence matrices form a tridiagonal pencil $(T_k,S_k)$,
\begin{align*}
T_k = \begin{bmatrix}
t_{1,1} & t_{1,2} &  &  &    \\
t_{2,1} & t_{2,2} & t_{2,3}   &  & \\
& t_{3,2} & t_{3,3} & \ddots  &\\
& & \ddots & \ddots   & t_{k-1,k} \\
& &  & t_{k,k-1} & t_{k,k} \\
& &  &  &t_{k+1,k} \\
\end{bmatrix},\\
S_k = \begin{bmatrix}
s_{1,1} &s_{1,2} &  &  &    \\
s_{2,1} & s_{2,2} & s_{2,3}   &  & \\
& s_{3,2} & s_{3,3} & \ddots  &\\
& & \ddots & \ddots   & s_{k-1,k} \\
& &  & s_{k,k-1} & s_{k,k} \\
& &  &  & s_{k+1,k} \\
\end{bmatrix},
\end{align*}
with $\frac{t_{i,i+1}}{s_{i,i+1}} = \xi_i$, $i = 1,2,\dots, k$ and  $\frac{t_{i+1,i}}{s_{i+1,i}} = \bar{\psi}_{i-1}$, $i = 2,3,\dots, k$.
These ratios are referred to as the poles of the pencil.
It is important to note that the ratio $t_{1,2}/s_{1,2}$ does not equal a pole.
A similar recurrence relation exists for $W_{k+1}$.
For $k=m$, the equation becomes $AV_m S = V_m T$ for $m\times m$ matrices $T,S$, obtained by truncating the last row of $T_m,S_m$, respectively.\\
Deriving this result will be too involved, but the underlying idea is the same as in the rational Arnoldi iteration. Details can be found in literature \cite{VBVBVa19}.
Theorem~\ref{theorem:sol_Lanczos} reveals that this iteration can be used to solve the TPIEP of Problem \ref{def:TPIEP}.

\begin{theorem}\label{theorem:sol_Lanczos}
	Let the matrices $V,W\in \mathbb{C}^{m\times m}$ be the bases and $T,S\in \mathbb{C}^{m\times m}$ the matrices of recurrences coefficients obtained by applying the rational Lanczos iteration for $\mathcal{K}_m(Z,v;\Xi)$ and $\mathcal{K}_m(Z^H,w;\varPsi)$, with $Z$, $v,w$ and $\Xi, \varPsi$ as in Problem~\ref{def:TPIEP}.
	Then these matrices solve Problem~\ref{def:TPIEP}, if the iteration does not break down.
\end{theorem}
\begin{proof}
	Since, under the no-breakdown assumption, $\mathcal{K}_m(Z,v;\Xi) = \mathcal{K}_m(Z^H,w;\varPsi) = \mathbb{C}^m$ and $W^H V = I$, we have $W = V^{-H}$. Thus $V^{-1}ZVS=T$, i.e., the spectrum of the pencil $(T,S)$ equals $\sigma(Z)$. The remainder of the proof follows immediately from the rational Lanczos iteration. With appropriate normalization $V e_1 = \frac{v}{\nu}$ and $We_1 = \frac{w}{\eta}$, with $\nu \bar{\eta}=w^Hv$.\qed
\end{proof}

\subsection{Uniqueness of the solution}
The uniqueness of the solution to Problem \ref{def:HPIEP} and \ref{def:TPIEP} is studied. It is shown that, for a fixed normalization, nested biorthonormal bases are essentially unique.
\begin{theorem}[Essential uniqueness of biorthonormal nested bases]\label{theorem:uniqueBases}
	Consider $A\in \mathbb{C}^{m\times m}$, $v,w\in  \mathbb{C}^m$ and rational Krylov subspaces
	\begin{align*}
	\mathcal{K}_m(A,v;\Xi), \quad
	\mathcal{K}_m(A^H,w;\varPsi).
	\end{align*}
	Then, under the no-breakdown assumption, a pair of biorthonormal nested bases $V,W\in \mathbb{C}^{m\times m}$ for these subspaces exists and this pair is essentially unique.
\end{theorem}
\begin{proof}
	Let $K_v,K_w$ denote the Krylov bases for $\mathcal{K}_n(A,v;\Xi)$ and $
	\mathcal{K}_n(A^H,w;\varPsi)$, respectively.
	The matrices $V,W$ form nested biorthonormal bases for $\mathcal{K}_n(A,v;\Xi)$, $\mathcal{K}_n(A^H,w;\varPsi)$, respectively.
	By Lemma \ref{theorem:momentLR} this implies the existence of upper triangular matrices $R_V, R_W$ such that
	\begin{align*}
	V = K_v R_V \quad \text{and} \quad W = K_w R_W.
	\end{align*}
	Consider also nested biorthonormal bases $\widetilde{V},\widetilde{W}$ which satisfy, for some upper triangular matrices $\widetilde{R}_V$, $\widetilde{R}_W$,
	\begin{align*}
	\widetilde{V} = K_v \widetilde{R}_V,\quad \widetilde{W} = K_w \widetilde{R}_W \quad \text{and} \quad \widetilde{W}^H \widetilde{V} = I.
	\end{align*}
	The relation of the upper triangular matrices to the moment matrix $M = K_w^H K_v$ proves the statement,
	\begin{equation*}
	W^H V = R_W^H K_w^H K_v R_V = R_W^H M R_V = I\Leftrightarrow M = R_W^{-H} R_V,
	\end{equation*}
	and similarly $M = \widetilde{R}_W^{-H} \widetilde{R}_V$.
	Both expressions are LR factorizations of the strongly nonsingular moment matrix $M$. Strong nonsingularity is implied by the assumption that no breakdowns occur.
	Since the LR factorization is essentially unique \cite[p.162]{HoJo85}, it follows that $\widetilde{R}_V = R_V D$ and $\widetilde{R}_W = R_W D$, where $D = \textrm{diag}(\alpha_1,\dots, \alpha_n)$, with $\vert \alpha_i\vert = 1$, if the same normalization is applied. This proves the essential uniqueness of the nested biorthonormal bases $V,W$. \qed
\end{proof}

Theorem \ref{theorem:uniqueBases} implies that the rational Lanczos iteration generates the same bases as would be obtained by the LR factorization of the corresponding moment matrix (or any other procedure constructing biorthonormal nested bases). The former generates, simultaneously with the bases, matrices of recurrence coefficients which adhere to the imposed structure, a tridiagonal pencil. This is not the case for the latter.
In the orthogonal case, the rational Arnoldi iteration, the above reasoning can be applied after setting $\mathcal{K}_m(A^H,w;\varPsi)=\mathcal{K}_m(A,v;\Xi)$ and $K_w=K_v$. An alternative proof for this case can be obtained based on the rational implicit Q theorem \cite{CaMeVa19_QZ}.

\section{Updating procedures}\label{sec:sol_update}
Updating procedures start from a known solution $B \in \mathbb{C}^{m \times m}$ to a structured IEP of size $m$ and construct a solution $\widetilde{B}\in \mathbb{C}^{(m+1)\times (m+1)}$ to a related IEP of size $m+1$, which is obtained by adding a node and weight(s) to the underlying bilinear form.
The goal is to devise such a procedure that is as efficient and numerically stable as possible.
Updating procedures consist of three steps
\begin{enumerate}
	\item Embed all matrices, belonging to the solution in $\mathbb{C}^{m\times m}$, in larger matrices belonging to $\mathbb{C}^{(m+1)\times (m+1)}$, denote embedded matrices by a hat.
	\item Enforce orthogonality with respect to the new weight vector(s), referred to as the orthogonality condition. This will disrupt the structure of the embedded matrix $\hat{B}\in\mathbb{C}^{(m+1)\times (m+1)}$.
	\item Enforce the correct structure on the embedded matrix $\hat{B}$, i.e., enforce the structure condition, to obtain $\widetilde{B}$, the solution to the larger IEP. 
\end{enumerate}
Problem \ref{def:HPIEP} (or \ref{def:TPIEP}) requires, aside from the new node and weight(s), also a new pole (or two poles) to fully characterize the structure of $\widetilde{B}$.
Section \ref{sec:HPIEP_update} proposes an updating procedure for Problem \ref{def:HPIEP} using unitary similarity transformations. 
This idea originates from a procedure to multiply J-fractions \cite{Ru63} and has been used to develop several procedures to solve IEPs linked to polynomials \cite{BuVB95,Re91,ReAmGr91b} and rational functions with prescribed, distinct poles \cite{VBFaGeMa05}.
The latter formulates the problem as a semiseparable-plus-diagonal IEP.
Our formulation is in the form of a Hessenberg pencil IEP.
Hessenberg pencils provide a more intuitive representation and allows more flexibility.\\
The idea also appears in different contexts but fulfilling a similar role, e.g., computing deflating subspaces for generalized eigenvalue problems \cite{VD81,CaMaVaWa20} and model order reduction \cite{OlRu06}. 
The updating procedure for the TPIEP from Problem \ref{def:TPIEP} must use nonunitary similarity transformations and is proposed in Section \ref{sec:TPIEP_Update}.
This problem occurs less often in literature.
Our interest in this case is triggered by the underlying short recurrence relation.	

\subsection{Updating Hessenberg pencil IEP}\label{sec:HPIEP_update}
An updating procedure for Problem \ref{def:HPIEP} is introduced. The updating relies on \emph{plane rotations} which are essentially $2\times2$ unitary matrices $P_i$, let $I_k$ denote the unit matrix of size $k\times k$, then
\begin{equation}\label{eq:plane_rotation}
P_i := \begin{bmatrix}
I_{i-1}     \\
&    \bar{a} &  & -\bar{b} \\
&     & I_{m-i} \\
& b  & & a \\
\end{bmatrix}, \text{ with parameters } a,b\in\mathbb{C} \text{ and } \bar{a}a+\bar{b}b = 1.
\end{equation}
The class of plane rotations $P_i$ is denoted by $\mathfrak{P}_i$.
The procedure starts from a unitary matrix $Q\in \mathbb{C}^{m\times m}$ and a proper Hessenberg pencil $(H,K)\in \mathbb{C}^{m\times m}$ that solve the HPIEP formulated in Problem \ref{def:HPIEP}.
A Hessenberg pencil is called \emph{proper} if the elements $h_{i+1,i}$ and $k_{i+1,i}$ are never simultaneously zero, for any $i< m$.
Let $Z = \textrm{diag}(\{z_i\}_{i=1}^m)$ denote the matrix of nodes, $v = \begin{bmatrix}
v_1 & v_2 & \dots & v_m
\end{bmatrix}^\top$ the weights and $\Xi = \{\xi_i \}_{i=1}^{m-1}$ poles of the considered problem. 
Then the following equalities are satisfied
\begin{equation}\label{eq:solution_originalHPIEP}
Q^H Z QK = H, \qquad Q^H v = \Vert v \Vert e_1.
\end{equation}
From this solution, construct the Hessenberg pencil $(\widetilde{H},\widetilde{K})\in \mathbb{C}^{(m+1)\times(m+1)}$ and orthonormal basis $\widetilde{Q}\in\mathbb{C}^{(m+1)\times (m+1)}$ which solve the HPIEP for $\widetilde{Z} := \textrm{diag}(\{z_i\}_{i=1}^{m+1})$, $\tilde{v} :=\begin{bmatrix}
v_1  & \dots & v_m & v_{m+1} 
\end{bmatrix}^\top$ and $\widetilde{\Xi} := \{\xi_i\}_{i=1}^{m}$.\\
First, the $m$-dimensional solution is embedded in a matrix in $\mathbb{C}^{(m+1)\times (m+1)}$ while preserving some key properties.
The unitary matrix $Q\in \mathbb{C}^{m\times m}$ is embedded while preserving its unitarity, and the pencil $(H,K)$ such that $\widetilde{Z}\widehat{Q}\widehat{K} = \widehat{Q}\widehat{H}$,
\begin{equation}\label{eq:embedded_HPIEP}
\widehat{Q} := \begin{bmatrix}
Q\\
& 1
\end{bmatrix}, \mkern6mu 	\widehat{H} := \begin{bmatrix}
H\\
& \hat{h}
\end{bmatrix},\mkern6mu
\widehat{K} := \begin{bmatrix}
K\\
& \hat{k}
\end{bmatrix},
\end{equation}
where $\hat{h}$ and $\hat{k}$ must satisfy $z_{m+1} \hat{k} =\hat{h}$.\\
The main result in this section is provided in Theorem \ref{theorem:solution_HPIEP} which states that $2m$ plane rotations suffice to construct the solution of an HPIEP of size $m+1$ starting from the solution of a related HPIEP of size $m$. The remainder of the section is dedicated to proving results that can be combined to obtain this theorem.
\begin{theorem}\label{theorem:solution_HPIEP}
	Let $Q\in \mathbb{C}^{m\times m}$, $(H,K)\in\mathbb{C}^{m\times m}$ be the solution to a HPIEP (Problem \ref{def:HPIEP}) of size $m$.
	And let $\widehat{Q}$, $\widehat{H}$ and $\widehat{K}$ denote the embedded matrices from \eqref{eq:embedded_HPIEP}.
	Then there exist $P_i, \dot{P}_i \in \mathfrak{P}_i$, $i=1,\dots, m$, such that
	\begin{equation*}
	\widetilde{H} = \prod_{l=m}^{1} P_l \widehat{H} \prod_{k=1}^{m}\dot{P}_k, \quad \widetilde{K} = \prod_{l=m}^{1} P_l \widehat{K} \prod_{k=1}^{m}\dot{P}_k,\quad \widetilde{Q} = \widehat{Q} \prod_{l=1}^{m}P_l^H
	\end{equation*}
	solve a HPIEP of size $m+1$, which is obtained by adding a node $z_{m+1}$, weight $v_{m+1}$ and pole $\xi_m$ to the original IEP. 
\end{theorem}
The embedded matrix $\widehat{Q}$ does not satisfy the orthogonality condition.
This condition will be enforced via multiplication with a suitable plane rotation $P_1\in \mathfrak{P}_1$, details are stated in Lemma \ref{lemma:P1}.

\begin{lemma}\label{lemma:P1}
	Let $\widehat{Q}$ be the embedded matrix from \eqref{eq:embedded_HPIEP}.
	Then there exists a $P_1 \in \mathfrak{P}_1$ such that $\widehat{Q} P_1^H$ satisfies the orthogonality condition $P_1\widehat{Q}^H \tilde{v} = \Vert \tilde{v} \Vert e_1$.
\end{lemma}
\begin{proof}
	Consider the original solution \eqref{eq:solution_originalHPIEP} and its embedding \eqref{eq:embedded_HPIEP}, then
	\begin{equation*}
	\widehat{Q}^H \tilde{v} =  \Vert {v} \Vert e_1 + v_{m+1}e_{m+1}.
	\end{equation*}
	Suitable parameters for $P_1$ are $a = \Vert v \Vert / \Vert \tilde{v}\Vert $   and $b =-v_{m+1}/\Vert \tilde{v}\Vert$, since 
	\begin{align*}
	\renewcommand{\arraystretch}{1.2}
	\begin{bmatrix}
	\bar{a} & - \bar{b}\\
	b & a
	\end{bmatrix} \begin{bmatrix}
	\Vert v \Vert \\
	v_{m+1}
	\end{bmatrix} = \begin{bmatrix}
	\Vert \tilde{v}\Vert\\
	0
	\end{bmatrix}.
	\end{align*}\qed
\end{proof}
In the new basis formed by columns of $\widehat{Q} P_1^H$, the matrix $\widetilde{Z}$ can be represented by the pencil $(P_1\widehat{H},P_1\widehat{K})$, this follows from
\begin{align*}
\widetilde{Z} \widehat{Q} P_1^H P_1 \widehat{K} &= \widehat{Q} P_1^H P_1\widehat{H}\\
\Leftrightarrow \widetilde{Z} \widehat{Q} P_1^H  \begin{bmatrix}
\times & \times & \dots & \times & \times & \times\\
\times & \times & \dots & \times & \times & 0\\
& \times & \dots & \times & \times & 0\\
& & \ddots & \vdots & \vdots & \vdots\\
& &  & \times & \times & 0\\
\times &\times &\dots  & \times & \times & \times\\
\end{bmatrix} &= \widehat{Q} P_1^H \begin{bmatrix}
\times & \times & \dots & \times & \times & \times\\
\times & \times & \dots & \times & \times & 0\\
& \times & \dots & \times & \times & 0\\
& & \ddots & \vdots & \vdots & \vdots\\
& &  & \times & \times & 0\\
\times &\times &\dots  & \times & \times & \times\\
\end{bmatrix}.
\end{align*}
This pencil deviates from Hessenberg structure in its last row, which is filled with generic nonzero elements, denoted by $\times$.\\
The following step is to restore the Hessenberg structure using unitary similarity transforms, i.e., enforce the structure condition. 
Lemma \ref{lemma:Hess_pencil} provides the details, but first Lemma \ref{lemma:NoBreakdown} is stated, which will guarantee that the element used for eliminating the elements in the last row never vanishes.
Thus, the procedure cannot break down.
\begin{lemma}\label{lemma:NoBreakdown}
	Consider a lower triangular matrix $L\in \mathbb{C}^{2\times 2}$ of full rank and nonsingular matrices $A,B\in \mathbb{C}^{2\times 2}$ such that $R = ALB$ is upper triangular. 
	Then $r_{(2,2)}\neq 0$.
\end{lemma}
\begin{proof}
	Since $A,B$ are nonsingular, $\textrm{rank}(R)=\textrm{rank}(ALB) = \textrm{rank}(L) = 2$ \cite[p.13]{HoJo85}, and since $R$ is upper triangular, it can only have full rank if $r_{2,2}\neq0$. \qed
\end{proof}

\begin{lemma}\label{lemma:Hess_pencil}
	Let $(\widehat{H},\widehat{K})\in \mathbb{C}^{(m+1)\times(m+1)}$ be a proper Hessenberg pencil and $P_1$ the plane rotation of Lemma \ref{lemma:P1}.
	Then there exist $P_l\in  \mathfrak{P}_l$, $l = 2,3,\dots, m$ and $\dot{P}_k\in  \mathfrak{P}_k$, $k = 1,2,\dots, m-1$ such that
	\begin{equation*}
	\prod_{l=m}^{1} P_l \widehat{H} \prod_{k=1}^{m-1} \dot{P}_k \textrm{ and }\prod_{l=m}^{1}P_l \widehat{K} \prod_{k=1}^{m-1} \dot{P}_k
	\end{equation*}
	are Hessenberg matrices where the subdiagonal ratios of their principal $m\times m$ submatrices, the poles from the pencil $(H,K)$, are preserved.
\end{lemma}
\begin{proof}
	Starting from $(P_1 \widehat{H},P_1\widehat{K})$, the elements of the last row are annihilated using plane rotations.
	Assume, without loss of generality, that the Hessenberg structure is restored up to column $i$, i.e., let $\alpha, \gamma$ denote the $(m+1,i+1)$th element, $\beta, \delta$ the $(i+2,i+1)$th element and $\epsilon,\eta$ the $(m+1,m+1)$th element of $\prod_{l=i+1}^{1} P_l \widehat{K} \prod_{k=1}^{i}\dot{P}_k$ and $\prod_{l=i+1}^{1} P_l \widehat{H} \prod_{k=1}^{i}\dot{P}_k$, respectively,
	\begin{align*}
	\prod_{l=i+1}^{1} P_l \widehat{H} \prod_{k=1}^{i}\dot{P}_k &=:  \begin{bmatrix}
	\widetilde{H}^{(i+1)\times i} & \boldsymbol{\times} & M_H & \boldsymbol{\times}\\
	& \delta e_1 & B_H & \boldsymbol{0}\\
	& \gamma& \boldsymbol{\times}^\top & \eta
	\end{bmatrix}, \\
	\prod_{l=i+1}^{1} P_l \widehat{K} \prod_{k=1}^{i}\dot{P}_k &=:  \begin{bmatrix}
	\widetilde{K}^{(i+1)\times i} & \boldsymbol{\times} & M_K & \boldsymbol{\times}\\
	& \beta e_1 & B_K & \boldsymbol{0}\\
	& \alpha& \boldsymbol{\times}^\top & \epsilon
	\end{bmatrix},
	\end{align*}
	In this expression $\widetilde{H}^{(i+1)\times i}$, $\widetilde{K}^{(i+1)\times i}$, respectively, are the principal submatrices of size ${(i+1)\times i}$ of the solution $(\widetilde{H}$,$\widetilde{K})$, $M_H, M_K\in  \mathbb{C}^{(i+1)\times (m-i-1)}$ are generically full matrices, and $B_H,B_K \in \mathbb{C}^{(m-i-1)\times (m-i-1)}$ are Hessenberg matrices. The zero vector, $\boldsymbol{0}$, and a vector containing generic nonzero elements, $\boldsymbol{\times}$, are assumed to be of appropriate size.\\		
	The elements $\alpha,\gamma$ must be eliminated, this is achieved in two steps.
	Since plane rotations are used, the relevant elements can be isolated in an equivalent $2\times 2$ problem, i.e. find parameters $a,b$ and $c,d$ appearing in $P_{i+2}$ and $\dot{P}_{i+1}$, respectively, such that
	\begin{align*}
	P_{i+2}^b H^b \dot{P}_{i+1}^b & :=
	\renewcommand{\arraystretch}{1.2}
	\begin{bmatrix}
	\bar{a} &-\bar{b}\\
	b & a
	\end{bmatrix} 
	\renewcommand{\arraystretch}{1.2}
	\begin{bmatrix}
	\delta & 0\\
	\gamma & \eta
	\end{bmatrix}
	\renewcommand{\arraystretch}{1.2}
	\begin{bmatrix}
	\bar{c} &-\bar{d}\\
	d & c
	\end{bmatrix} =
	\renewcommand{\arraystretch}{1.2}
	\begin{bmatrix}
	h & \times\\
	0 & \times
	\end{bmatrix},\\
	P_{i+2}^b K^b \dot{P}_{i+1}^b &:= 
	\renewcommand{\arraystretch}{1.2}
	\begin{bmatrix}
	\bar{a} &-\bar{b}\\
	b & a
	\end{bmatrix} \begin{bmatrix}
	\beta & 0\\
	\alpha & \epsilon
	\end{bmatrix}
	\renewcommand{\arraystretch}{1.2}
	\begin{bmatrix}
	\bar{c} &-\bar{d}\\
	d & c
	\end{bmatrix} = 
	\renewcommand{\arraystretch}{1.2}
	\begin{bmatrix}
	k & \times\\
	0 & \times
	\end{bmatrix}
	\end{align*} 
	with $\frac{h}{k} = \frac{\delta}{\beta} = \xi_{i+1}$ and the superscript $b$, for \emph{block}.
	First $\dot{P}_{i+1}^b$ is constructed such that $M \dot{P}_{i+1}^b e_1 = \begin{bmatrix}
	0 & 0
	\end{bmatrix}^\top$ with
	\begin{equation*}
	M = \delta K^b - \beta H^b = \begin{bmatrix}
	0 & 0\\
	\times & \times
	\end{bmatrix}.
	\end{equation*}
	This results in $\textrm{rank}\left(\begin{bmatrix}
	H^b\dot{P}_{i+1}^b e_1 &K^b\dot{P}_{i+1}^b e_1
	\end{bmatrix}\right)=1$, i.e., first columns of $H^b \dot{P}_{i+1}^b$ and $K^b \dot{P}_{i+1}^b$ are colinear (also called parallel).
	Second, $P_{i+2}^b$ is chosen to make $H^b\dot{P}_{i+1}^b$ (or $K^b\dot{P}_{i+1}^b$) upper triangular, i.e., $P_{i+2}^b H^b\dot{P}_{i+1}^b e_1 = \begin{bmatrix}
	k & 0
	\end{bmatrix}^\top$. Thanks to the colinearity, the same plane rotation $P_{i+2}^b$ results in $P_{i+2}^b H^b\dot{P}_{i+1}^b= \begin{bmatrix}
	h & 0
	\end{bmatrix}^\top$.
	For some nonzero constant $u$, $h=u\delta$ and $k = u\beta$ and therefore the ratio of subdiagonal elements $\frac{h}{k} = \frac{u\delta}{u\beta} = \xi_{i+1}$ is preserved.
	By Lemma \ref{lemma:NoBreakdown} the elements $e_2^\top P_{i+2}^b H^b\dot{P}_{i+1}^b e_1$ and $e_2^\top P_{i+2}^b K^b\dot{P}_{i+1}^b e_1$ are nonzero. Hence, this process can be repeated until $i=m-2$.
	If $\delta = 0$ (or $\beta=0$, never both), then $\textrm{rank}(H^b) = 1$ (or $\textrm{rank}(K^b) = 1$)  and the lemma does not hold, however straightforward computation shows that, in this case, $e_2^\top P_{i+2}^b H^b\dot{P}_{i+1}^b e_1 \neq 0$ (or $e_2^\top P_{i+2}^b K^b\dot{P}_{i+1}^b e_1\neq 0$). \qed
\end{proof}
Note that $\dot{P}_l$, $l = 1,2,\dots,m-1$ do not have to be unitary as in the above proof, they only have to be nonsingular for the proof to hold. However, unitary $\dot{P}_l$ are preferred for numerical computations.
The required structure, a Hessenberg pencil, is obtained by the matrices in Lemma \ref{lemma:Hess_pencil}.	
However the ratio of their last subdiagonal elements will not necessarily equal the new pole $\xi_m$. 
Lemma \ref{lemma:poles_Hess} shows that the correct pole can be introduced.
\begin{lemma}\label{lemma:poles_Hess}
	Consider $\xi_m\in\mathbb{C}$ and let $\mu, \nu$, and $\epsilon,\eta$ denote, respectively, the elements on positions $(m+1,m)$ and $(m+1,m+1)$ in the matrices from Lemma~\ref{lemma:Hess_pencil},
	\begin{align}
	\prod_{l=m}^{1} P_l \widehat{H} \prod_{k=1}^{m-1} \dot{P}_k &=  \begin{bmatrix}
	\widetilde{H}^{m\times (m-1)} &\boldsymbol{\times}& \boldsymbol{\times}\\
	& \mu& \epsilon
	\end{bmatrix},\nonumber\\
	\prod_{l=m}^{1}P_l \widehat{K} \prod_{k=1}^{m-1} \dot{P}_k &= \begin{bmatrix}
	\widetilde{K}^{m\times (m-1)} &\boldsymbol{\times}& \boldsymbol{\times}\\
	& \nu& \eta
	\end{bmatrix}\label{eq:almost_HPIEP}.
	\end{align}
	Then $\dot{P}_m\in\mathfrak{P}_m$, with parameters $\dot{a},\dot{b}$, exists such that 
	\begin{align*}
	\begin{bmatrix}
	\mu& \epsilon
	\end{bmatrix} \begin{bmatrix}
	\bar{c}\\
	d 
	\end{bmatrix} =	\hat{h} \quad \textrm{and}\quad
	\begin{bmatrix}
	\nu& \eta
	\end{bmatrix} \begin{bmatrix}
	\bar{c}\\
	d 
	\end{bmatrix} = \hat{k}, \quad \text{with } \frac{\hat{h}}{\hat{k}} = \xi_{m}.
	\end{align*}
\end{lemma}
\begin{proof}
	Since $\epsilon$ and $\eta$ are both nonzero by Lemma \ref{lemma:NoBreakdown} and the matrices from \eqref{eq:almost_HPIEP} form a proper Hessenberg pencil, suitable $c$ and $d$ exist. \qed
\end{proof}

\subsection{Updating tridiagonal pencil IEP}\label{sec:TPIEP_Update}
An updating procedure for Problem \ref{def:TPIEP} is formulated. The ideas are the same as in Section \ref{sec:HPIEP_update}. However, it is no longer possible to use unitary similarity transformations. Instead of plane rotations, \emph{eliminators} will be used, which are essentially $2\times 2$ triangular matrices.
Let $\mathfrak{L}_i$ and $\mathfrak{R}_i$, respectively, denote the classes of lower and upper triangular eliminators, i.e., classes composed, respectively, of matrices of the form
\begin{equation}\label{eq:eliminators}
L_i := \begin{bmatrix}
I_{i-1}     \\
&    1 &  &  \\
&     & I_{m-i} \\
& a_i  & & 1 \\
\end{bmatrix}
\text{ and } 
R_i := \begin{bmatrix}
I_{i-1}     \\
&    1 &  &  b_i\\
&     & I_{m-i} \\
&  & & 1 \\
\end{bmatrix},
\end{equation} 
with parameters $a_i,b_i\in \mathbb{C}$.
The elements on position $(i,i)$ and $(m+1,m+1)$ allow other choices than the value $1$, and an appropriate choice for them is paramount to the development of an effective numerical procedure.
However for clarity of the following exposition they are fixed at the value $1$.
Moreover, in this updating procedure, breakdowns can occur. 
For simplicity of the exposition we assume that no breakdowns occur.
To initiate the updating procedure, suppose we possess the solution to the tridiagonal pencil IEP for nodes $Z = \textrm{diag}(\{z_i \}_{i=1}^{m})$, weights $v = \begin{bmatrix}
v_1 & v_2 & \dots & v_m 
\end{bmatrix}^\top$, $w = \begin{bmatrix}
w_2 & w_2 & \dots & w_m 
\end{bmatrix}^\top$ and poles $\Xi = \{\xi_i\}_{i=1}^{m-1}$, $\Psi = \{\psi_i\}_{i=1}^{m-1}$. 
This solution consists of biorthonormal matrices $V,W\in \mathbb{C}^{m\times m}$ and a tridiagonal pencil $(T,S)\in \mathbb{C}^{m\times m}$ satisfying, for $w^H v =  \bar{\eta}\nu$,
\begin{equation}\label{eq:TPIEP_original}
W^H Z V S = T, \quad V^H w= \eta e_1, \quad W^H v= \nu e_1.
\end{equation}
Next, a node $z_{m+1}$, weights $v_{m+1},w_{m+1}$ and poles $\xi_m,\psi_m$ are added to the original problem.
The new problem consists of nodes $\widetilde{Z} = \textrm{diag}(\{z_i \}_{i=1}^{m+1})$, weights $\tilde{v} = \begin{bmatrix}
v_1 & v_2 & \dots & v_{m+1}
\end{bmatrix}^\top$, $\tilde{w} = \begin{bmatrix}
w_1 & w_2 & \dots & w_{m+1}
\end{bmatrix}^\top$ and poles $\widetilde{\Xi} = \{\xi_i \}_{i=1}^{m}$, $\widetilde{\Psi} = \{\psi_i \}_{i=1}^{m}$.
The biorthonormal matrices $V,W\in \mathbb{C}^{m\times m}$ are embedded while preserving their biorthonormality
\begin{equation}\label{eq:embeddedVW}
\widehat{V} := \begin{bmatrix}
V\\
& 1
\end{bmatrix}, \qquad
\widehat{W} := \begin{bmatrix}
W\\
& 1
\end{bmatrix}.
\end{equation}
And for $\hat{s}_{m+1}$ and $\hat{t}_{m+1}$ satisfying $z_{m+1} \hat{s}_{m+1} = \hat{t}_{m+1}$, define $\widehat{T}$ and $\widehat{S}$ as 
\begin{equation}\label{eq:embeddedTS}
\widetilde{Z} \widehat{V} \underbrace{\begin{bmatrix}
	S\\
	& \hat{s}_{m+1}
	\end{bmatrix}}_{=:\widehat{S}} = \widehat{V} \underbrace{\begin{bmatrix}
	T\\
	& \hat{t}_{m+1}
	\end{bmatrix}}_{=:\widehat{T}}.
\end{equation}
Theorem \ref{theorem:TPIEP_solution} states how a solution can be efficiently obtained from the embedded matrices \eqref{eq:embeddedVW}, \eqref{eq:embeddedTS} using eliminators.
The remainder of this section is dedicated to providing the components required to prove this theorem.
\begin{theorem}\label{theorem:TPIEP_solution}
	Let $V, W\in \mathbb{C}^{m\times m}$, $(T,S)\in \mathbb{C}^{m\times m}$ be the solution to a tridiagonal pencil IEP (Problem \ref{def:TPIEP}) of size $m$ and let $\widehat{V},\widehat{W}, \widehat{T},\widehat{S}$ denote the corresponding embedded matrices \eqref{eq:embeddedVW}, \eqref{eq:embeddedTS}. Then there exist a diagonal matrix $D$, a nonsingular matrix $C$, eliminators $L_i\in \mathfrak{L}_i$, $R_i\in \mathfrak{R}_i$ , $i=1,2,\dots, m$ and $\dot{L}_j\in \mathfrak{L}_j$, $\dot{R}_j\in \mathfrak{R}_j$, $j=1,2,\dots, m$ such that
	\begin{align*}
	\widetilde{T} &= \left(\prod_{k=m}^{2}R_k \dot{L}_{k} \right) R_1 D \dot{L}_1 \widehat{T} C L_1  \left(\prod_{k=2}^{m}\dot{R}_{k} L_k  \right),\\
	\widetilde{S} &= \left(\prod_{k=m}^{2}R_k \dot{L}_{k} \right) R_1 D \dot{L}_1 \widehat{S} C L_1 \left(\prod_{k=2}^{m}\dot{R}_{k} L_k  \right),\\
	\widetilde{V} &= \widehat{V} \dot{L}_1 D^{-1} R_1^{-1} \left(\prod_{k=2}^{m}\dot{L}_{k}^{-1} R_k^{-1} \right) \text{ and}\quad 		\widetilde{W} = \widehat{W} \dot{L}_{1}^{H} D^H R_1^{H} \left(\prod_{k=2}^{m} \dot{L}_{k}^{H} R_k^{H} \right),
	\end{align*}
	solve the IEP of size $m+1$.
	This IEP is obtained by adding a node $z_{m+1}$, weights $v_{m+1},w_{m+1}$ and poles $\xi_m,\psi_m$ to the original IEP of dimension $m$.
\end{theorem}	
The matrices appearing in Theorem \ref{theorem:TPIEP_solution} can be determined by taking the orthogonality and structure conditions into account.
The embedded matrices $\widehat{V},\widehat{W}$ are, in general, not orthogonal to the weight vectors, i.e., for the new weights $v_{m+1}, w_{m+1} \neq 0$,
\begin{align}\label{eq:embed_weight_orthog}
\widehat{V}^H \tilde{w} =  \eta e_1 + w_{m+1} e_{m+1},\qquad
\widehat{W}^H \tilde{v} =  \nu e_1 + v_{m+1} e_{m+1}.
\end{align} 
Lemma \ref{lemma:update_TP_C} states that, starting from the biorthogonal bases $\widehat{V},\widehat{W}$, another pair of biorthogonal bases, that satisfies the orthogonality condition, can be constructed efficiently.
\begin{lemma}\label{lemma:update_TP_C}
	Let $\widehat{V},\widehat{W}$ be the embedded matrices \eqref{eq:embeddedVW} and $\tilde{v},\tilde{w}$ the weight vectors of the considered IEP of size $m+1$.
	Then there exist a diagonal matrix $D$, $\dot{L}_1\in \mathfrak{L}_1$ and $R_1 \in \mathfrak{R}_1$
	such that $R_1^{-H} D^{-H}\dot{L}_1^{-H} \widehat{V}^H \tilde{w} = \hat{\eta} e_1$ and $R_1 D^H \dot{L}_1 \widehat{W}^H \tilde{v} = \hat{\nu} e_1$, with $\hat{\nu}\bar{\hat{\eta}}=\tilde{w}^H \tilde{v}$.
\end{lemma}
\begin{proof}
	Clearly $\dot{L}_1$ can be chosen to eliminate the last element of $\widehat{V}^H \tilde{w}$ and $R_1$ to eliminate the last element of $\widehat{W}^H \tilde{v}$. The diagonal matrix $D$ scales such that $w_{m+1}$ is introduced correctly, relative to the weights in $w$. \qed
\end{proof}

The new bases lead to a representation $(R_1 D \dot{L}_1 \widehat{T},R_1 D \dot{L}_1\widehat{S})$ of $\widetilde{Z}$ which is no longer a tridiagonal pencil
\begin{equation} \label{eq:disturbed_TP}
\widetilde{Z} \widehat{V} \dot{L}_1^{-1} D^{-1} R_1^{-1} 
\left[
\renewcommand{\arraystretch}{1.2}
\begin{array}{ccc|c}
& & & \times\\
& \bar{S} & & 0\\
& & & \boldsymbol{0}\\ \hline
\times & \times & \boldsymbol{0}^\top & \times
\end{array}
\right]
= \widehat{V} \dot{L}_1^{-1} D^{-1} R_1^{-1} 
\left[
\renewcommand{\arraystretch}{1.2}
\begin{array}{ccc|c}
& & & \times\\
& \bar{T} & & 0\\
& & & \boldsymbol{0}\\ \hline
\times & \times & \boldsymbol{0}^\top & \times
\end{array}
\right].
\end{equation}
The matrices $\bar{T}$ and $\bar{S}$ are tridiagonal and only differ from $T$ and $S$, respectively, in the first two rows.
Using eliminators, the tridiagonal structure of the pencil will be enforced, Lemma \ref{lemma:recover_TP} provides the details.
\begin{lemma}\label{lemma:recover_TP}
	Let $(\widehat{T},\widehat{S})$ denote the embedded pencil \eqref{eq:embeddedTS} and $R_1$, $\dot{L}_1$ the matrices from Lemma \ref{lemma:update_TP_C}.
	Then, under the assumption that no breakdown occurs, there exist a nonsingular matrix $C$, eliminators $L_i\in \mathfrak{L}_i$, $R_i\in \mathfrak{R}_i$, for $i=1,2,\dots, m-1$ and $\dot{L}_j\in \mathfrak{L}_j$, $\dot{R}_j\in \mathfrak{R}_j$, $j=2,3,\dots, m$ such that the following matrices have tridiagonal structure:
	\begin{align}
	\dot{T} := \dot{L}_m \left(\prod_{k=1}^{m-1} R_k \dot{L}_k \right)  \widehat{T} C L_1 \left(\prod_{k=2}^{m-1} \dot{R}_k L_k \right) \dot{R}_m,\nonumber\\
	\dot{S} := \dot{L}_m \left(\prod_{k=1}^{m-1} R_k \dot{L}_k \right)  \widehat{S} C L_1 \left(\prod_{k=2}^{m-1} \dot{R}_k L_k \right) \dot{R}_m.\label{eq:red_to_tridiag}
	\end{align}
	Furthermore, the poles appearing in the pencil $(\widehat{T},\widehat{S})$ are preserved under these operations.
\end{lemma}
\begin{proof}
	The proof is by induction. 
	The first step, $i=1$, differs from the general iteration ($i\geq 2$).
	Consider the matrices $\dot{L}_1,R_1$ from Lemma \ref{lemma:update_TP_C}, and denote their parameter by $\dot{a}_1$, $b_1$, respectively.
	Then the pencil which must be reduced to tridiagonal form is, for some $\hat{t}^{(1)}, \hat{s}^{(1)}$ satisfying $\hat{t}^{(1)} z_{m+1} = \hat{s}^{(1)}$,
	\begin{align*}
	\widehat{T}_1 := R_1 \dot{L}_1 \widehat{T} =\left[
	\renewcommand{\arraystretch}{1.2}
	\begin{array}{c|c|c}
	(1+\dot{a}_1 b_1)t_{11} & \begin{matrix} (1+\dot{a}_1 b_1)t_{12}  &\boldsymbol{0} \end{matrix} & b_1 \hat{t}^{(1)}\\ \hline
	\hat{t}_{21} &  & 0\\
	\boldsymbol{0}& \bar{T}^{(m-1)} & \boldsymbol{0}\\ 
	& &  \\ \hline
	\dot{a}_1 t_{11} & \begin{matrix} \dot{a}_1 t_{12}  & \hfill\hspace{0.5cm} \boldsymbol{0}^\top \end{matrix} & \hat{t}^{(1)}\\
	\end{array}
	\right],\\
	\widehat{S}_1 := R_1 \dot{L}_1 \widehat{S} =\left[
	\renewcommand{\arraystretch}{1.2}
	\begin{array}{c|c|c}
	(1+\dot{a}_1 b_1)s_{11} & \begin{matrix} (1+\dot{a}_1 b_1)s_{12}  &\boldsymbol{0} \end{matrix} & b_1 \hat{s}^{(1)}\\ \hline
	\hat{s}_{21} &  & 0\\
	\boldsymbol{0}& \bar{S}^{(m-1)} & \boldsymbol{0}\\ 
	& &  \\ \hline
	\dot{a}_1 s_{11} & \begin{matrix} \dot{a}_1 s_{12}  & \hfill\hspace{0.5cm} \boldsymbol{0}^\top \end{matrix} & \hat{s}^{(1)}\\
	\end{array}
	\right],
	\end{align*}	
	where $\bar{T}^{(m-1)}$ and $\bar{S}^{(m-1)}$ denote, respectively, the principal trailing submatrix of size $(m-1)\times (m-1)$ of $\bar{T}$ and $\bar{S}$ from $\eqref{eq:TPIEP_original}$.
	Note in the above equation that the matrices are very similar, therefore, we will only explicitly write down $\widehat{T}_i$ and omit $\widehat{S}_i$.
	Similarly as in the proof of Theorem \ref{theorem:solution_HPIEP}, the annihilation of the first element in the last row and last column is performed in two steps.
	The first step creates colinearity between relevant elements in the pencil and the second step eliminates the first elements in the last row (and last column) simultaneously.\\
	First, create colinearity, i.e., find suitable $L_1\in \mathfrak{L}_1$ with parameter $a_1$ and $C := \begin{bmatrix}
	1& c\\
	&1\\
	& & I_{m-1}
	\end{bmatrix}$ which act on $\widehat{T}_1$ as follows
	\begin{equation*}
	\widehat{T}_1 C L_1 =\left[
	\renewcommand{\arraystretch}{1.2}
	\begin{array}{c|c|c}
	\tilde{t}_{11} & \begin{matrix} \tilde{t}_{12} &\boldsymbol{0} \end{matrix} & b_1 \hat{t}^{(1)}\\ \hline
	\tilde{t}_{21} &  & 0\\
	\boldsymbol{0}& \bar{T}^{(m-1)} + c\hat{t}_{21} e_1 e_1^\top & \boldsymbol{0}\\ 
	& &  \\ \hline
	\dot{a}_1 t_{11} + a_1 \hat{t}^{(1)} & \begin{matrix} \dot{a}_1 (t_{12} + ct_{11})  & \hfill\hspace{0.5cm} \boldsymbol{0}^\top \end{matrix} & \hat{t}^{(1)}\\
	\end{array}
	\right],
	\end{equation*}
	where $\tilde{t}_{11}:=(1+\dot{a}_1 b_1)t_{11} + a_1 b_1 \hat{t}^{(1)}$, $\tilde{t}_{12}:=(1+\dot{a}_1 b_1)(t_{12}+ct_{11})$, $\tilde{t}_{21}:=\hat{t}_{21}$.
	The resulting matrix should satisfy the colinearity conditions
	\begin{equation*}
	\textrm{rank}\left(\begin{bmatrix}
	\tilde{t}_{12} & b_1 \hat{t}^{(1)}\\
	\tilde{s}_{12} & b_1 \hat{s}^{(1)}
	\end{bmatrix} \right) = 1,\quad
	\textrm{rank}\left(\begin{bmatrix}
	\tilde{t}_{21} & \tilde{s}_{21}\\
	\dot{a}_1 t_{11} + a_1 \hat{t}^{(1)} &  \dot{a}_1 s_{11} + a_1 \hat{s}^{(1)}
	\end{bmatrix} \right) = 1,
	\end{equation*}
	this can be achieved by choosing appropriate $c$ and $a_1$.
	Next, $\dot{L}_2\in \mathfrak{L}_2$ and $\dot{R}_2\in \mathfrak{R}_2$ eliminate these colinear elements in the last row and last column, i.e., 
	\begin{equation*}
	\dot{L}_2\widehat{T}_1 C L_1 \dot{R}_2 = 	\left[
	\renewcommand{\arraystretch}{1.2}
	\begin{array}{cccc|c}
	\tilde{t}_{11} & 
	\multicolumn{1}{c|}{\tilde{t}_{12}} 
	& 0 & \boldsymbol{0}^\top & 0\\
	\cline{2-5}
	\multicolumn{1}{c|}{\tilde{t}_{21}} & & & & \hat{t}_{2,m+1}\\
	\cline{1-1} 
	\multicolumn{1}{c|}{0} &  \multicolumn{3}{c|}{\bar{T}^{(m-1)} + c\hat{t}_{21} e_1 e_1^\top} & \dot{b}_2 t_{32}\\
	\multicolumn{1}{c|}{\boldsymbol{0}} & & &  & \boldsymbol{0}\\
	\cline{1-5}
	\multicolumn{1}{c|}{0}& \hat{t}_{m+1,2} & \dot{a}_2 t_{23} & \boldsymbol{0}^\top & \hat{t}^{(2)}
	\end{array}
	\right]=:\widehat{T}_2,
	\end{equation*}
	where $\hat{t}_{m+1,2} := \dot{a}_1(t_{12}+ct_{11}) + \dot{a_2}(\hat{t}_{22} + c \hat{t}_{21})$, $\hat{t}_{2,m+1} := \dot{b}_1(\hat{t}_{22}+c\hat{t}_{21})$ and $\hat{t}^{(2)} := (1+\dot{a}_2 b_1)\hat{t}^{(1)} + \dot{b}_2 \dot{a}_1 (t_{12}+c t_{11}) + \dot{b}_2 \dot{a}_2 (\hat{t}_{22}+ c t_{21})$.
	By the colinearity, the same holds for $\widehat{S}_2$.
	This shows that the initial step can be performed using the matrices $C,L_1,\dot{L}_2$ and $\dot{R}_2$.
	Note that this $C$ is only required in the first step, in the subsequent steps it will be replaced by a matrix in $\mathfrak{R}_i$.
	Under the induction hypothesis, we have
	\begin{equation*}
	\widehat{T}_i = \dot{L}_i R_{i-1} \widehat{T}_{i-1} L_{i-1} \dot{R}_i  = \left[
	\renewcommand{\arraystretch}{1.2}
	\begin{array}{cccc|c}
	\widetilde{T}^{(i-1)} & \multicolumn{1}{c|}{\tilde{t}_{i-1,i} e_i}&  &  & \boldsymbol{0}\\
	\cline{2-5}
	\multicolumn{1}{c|}{\tilde{t}_{i,i-1}e_i^\top} & & & & \hat{t}_{i,m+1}\\
	\cline{1-1} 
	&  \multicolumn{3}{|c|}{T^{(m-i+1)}} & \dot{b}_i t_{i+1,i}\\
	\multicolumn{1}{c|}{} & & &  & \boldsymbol{0}\\
	\cline{1-5}
	\multicolumn{1}{c|}{\boldsymbol{0}}& \hat{t}_{m+1,i} & \dot{a}_i t_{i,i+1} & \boldsymbol{0}^\top & \hat{t}^{(i)}
	\end{array}
	\right],
	\end{equation*}
	where $T^{(m-i+1)}\in \mathbb{C}^{(m-i+1)\times (m-i+1)}$, the principal trailing submatrix of $T$ \eqref{eq:TPIEP_original} and $\widetilde{T}^{(i-1)}$ is the $(i-1)\times (i-1)$ leading principal submatrix of the solution $\widetilde{T}$ to the TPIEP.
	All the action takes place in the $(m-i+1)\times (m-i+1)$ principal trailing submatrix of $\widehat{T}_i$.
	For the proof to hold, $\widehat{T}_{i+1} = \dot{L}_{i+1} R_{i} \widehat{T}_{i} L_{i} \dot{R}_{i+1}$ and $\widehat{S}_{i+1}= \dot{L}_{i+1} R_{i} \widehat{S}_{i} L_{i} \dot{R}_{i+1}$ must create zeros on positions $(m+1,i)$ and $(i,m+1)$.
	The first step enforcing colinearity, determining $L_i \in \mathfrak{L}_i$ and $R_i \in \mathfrak{R}_i$, can be elegantly formulated using
	\begin{align*}
	M_l &:= s_{i+1,i}\begin{bmatrix}
	t_{i+1,i} & \dot{b}_i t_{i+1,i}\\
	\hat{t}_{m+1,i} & \hat{t}^{(i)}
	\end{bmatrix} - t_{i+1,i} \begin{bmatrix}
	s_{i+1,i} & \dot{b}_i s_{i+1,i}\\
	\hat{s}_{m+1,i} & \hat{s}^{(i)}
	\end{bmatrix} = \begin{bmatrix}
	0 &0 \\
	\times & \times
	\end{bmatrix},\\
	M_r &:= s_{i,i+1}\begin{bmatrix}
	t_{i,i+1} & \hat{t}_{i,m+1}\\
	\dot{a}_i t_{i,i+1} & \hat{t}^{(i)}
	\end{bmatrix} - t_{i+1,i} \begin{bmatrix}
	s_{i,i+1} & \hat{s}_{i,m+1}\\
	\dot{a}_i s_{i,i+1} & \hat{s}^{(i)}
	\end{bmatrix} = \begin{bmatrix}
	0 & \times \\
	0 & \times
	\end{bmatrix}.
	\end{align*}
	They are constructed such that, isolating their active part in $L^b_i$ and $R^b_i$, they create rank one matrices
	\begin{equation*}
	M_l L^b_i = \begin{bmatrix}
	0 & 0\\
	0 & \times
	\end{bmatrix} \text{ and }
	R^b_i M_r = \begin{bmatrix}
	0 &0 \\
	0 & \times
	\end{bmatrix}.
	\end{equation*}
	Clearly, appropriate $\dot{L}_{i+1}\in \mathfrak{L}_{i+1}$ and $\dot{R}_{i+1}\in \mathfrak{R}_{i+1}$ can be found such that
	\begin{align*}
	\widehat{T}_{i+1} &= \dot{L}_{i+1} R_{i} \widehat{T}_{i} L_{i} \dot{R}_{i+1}  \\
	&= \left[
	\renewcommand{\arraystretch}{1.2}
	\begin{array}{cccc|c}
	\widetilde{T}^{(i)} & \multicolumn{1}{c|}{\tilde{t}_{i,i+1} e_{i+1}}&  &  & \boldsymbol{0}\\
	\cline{2-5}
	\multicolumn{1}{c|}{\tilde{t}_{i+1,i}e_{i+1}^\top}  & & & & \hat{t}_{i+1,m+1}\\
	\cline{1-1} 
	\multicolumn{1}{c|}{}&  \multicolumn{3}{c|}{T^{(m-i)}} & \dot{b}_{i+1}t_{i+2,i+1}\\
	\multicolumn{1}{c|}{}& & &  & \boldsymbol{0}\\
	\cline{1-5}
	\multicolumn{1}{c|}{\boldsymbol{0}} & \hat{t}_{m+1,i+1} & \dot{a}_{i+1} t_{i+1,i+2} & \boldsymbol{0}^\top & \hat{t}^{(i+1)}
	\end{array}
	\right],
	\end{align*}
	with $\hat{t}_{m+1,i+1}=\dot{a}_i t_{i,i+1} + \dot{a}_{i+1}t_{i+1,i+1}$, $\hat{t}^{(i+1)}=\hat{t}^{(i)} + \dot{a}_{i+1} \dot{b}_i t_{i+1,i} + \dot{b}_{i+1} \dot{a}_i t_{i,i+1}$ and $\hat{t}_{i+1,m+1} = \dot{b}_i t_{i+1,i} + \dot{b}_{i+1} t_{i+1,i+1}$.
	Furthermore, $\tilde{t}_{i,i} = t_{i,i} + b_i \hat{t}_{m+1,i} + a_i \hat{t}_{i,m+1}$, $\tilde{t}_{i,i+1} =(1+\dot{a}_i b_i) t_{i,i+1}$ and $\tilde{t}_{i+1,i} = (1+\dot{b}_i a_i)t_{i+1,i}$.
	Same holds for $\widehat{S}_{i+1}$ thanks to the colinearity.
	Hence, the poles are preserved, $ \tilde{t}_{i+1,i}/\tilde{s}_{i+1,i} = t_{i+1,i}/s_{i+1,i} = \xi_i$, $i=1,2,\dots, m-1$ and $ \tilde{t}_{i,i+1}/\tilde{s}_{i,i+1} = t_{i,i+1}/s_{i,i+1} = \bar{\psi}_{i-1}$, $i=2,3,\dots, m-1$.\\
	This can be continued for $i< m$, thereby obtaining at $i=m-1$ a tridiagonal pencil $(\dot{T},\dot{S}) := (\widehat{T}_m,\widehat{S}_m)$.
	Whenever $\hat{t}^{(i)}$ (or $\hat{s}^{(i)}$) vanishes for some $i$, the above transformations cannot be determined, i.e., a breakdown occurs. This situation is excluded by the assumption that no breakdowns occur. \qed
\end{proof}

Lemma \ref{lemma:recover_TP} shows that a tridiagonal pencil $(\dot{T},\dot{S})$ with eigenvalues $z_i$ can be constructed while preserving its poles, i.e., $\frac{\dot{t}_{i+1,i}}{\dot{s}_{i+1,i}} = \xi_i$, $i = 1,2,\dots, m-1$ and $\frac{\dot{t}_{i,i+1}}{\dot{s}_{i,i+1}} = \bar{\psi}_{i-1}$, $i = 2,3,\dots, m-1$.
However, in general $\frac{\dot{t}_{m+1,m}}{\dot{s}_{m+1,m}} \neq \xi_m$ and $\frac{\dot{t}_{m,m+1}}{\dot{s}_{m,m+1}} \neq \bar{\psi}_{m-1}$, which are the poles belonging to the rational functions that are added to the space.\\
The new poles must be introduced in the pencil without disturbing the tridiagonal structure, Lemma \ref{lemma:poles_tridiag} provides the details.
\begin{lemma}\label{lemma:poles_tridiag}
	Consider a tridiagonal pencil $(\dot{T},\dot{S})$, where $\dot{t}_{m+1,m+1},\dot{s}_{m+1,m+1}\neq 0$.
	Premultiplication with $R_m\in \mathfrak{R}_m$ and postmultiplication with $L_m\in \mathfrak{L}_m$ suffice to alter, respectively, the ratio of elements $(m,m+1)$ and $(m+1,m)$ of $(\dot{T},\dot{S})$ to any values $\psi_{m-1}\in\bar{\mathbb{C}}$ and $\xi_m\in\bar{\mathbb{C}}$.
\end{lemma}
\begin{proof}
	Only the $2\times2$ trailing principal submatrices of $(\dot{T},\dot{S})$ are altered by the transformation $(R_m \dot{T} L_m, R_m \dot{S} L_m)$, denoted compactly by
	\begin{align*}
	\begin{bmatrix}
	1 & b_m \\
	& 1
	\end{bmatrix} \begin{bmatrix}
	\dot{t}_{m,m} & \dot{t}_{m,m+1}\\
	\dot{t}_{m+1,m} & \dot{t}_{m+1,m+1}
	\end{bmatrix}
	\begin{bmatrix}
	1 & \\
	a_m &1
	\end{bmatrix} &= \begin{bmatrix}
	\times & \alpha\\
	\mu &  \times
	\end{bmatrix}\\
	\begin{bmatrix}
	1 & b_m \\
	& 1
	\end{bmatrix} \begin{bmatrix}
	\dot{s}_{m,m} & \dot{s}_{m,m+1}\\
	\dot{s}_{m+1,m} & \dot{s}_{m+1,m+1}
	\end{bmatrix}
	\begin{bmatrix}
	1 & \\
	a_m &1
	\end{bmatrix} &= \begin{bmatrix}
	\times & \beta\\
	\nu &  \times
	\end{bmatrix}.
	\end{align*}
	The formation of the product on the left-hand side shows that, under the assumption $\dot{t}_{m+1,m+1}\neq 0$ and $\dot{s}_{m+1,m+1}\neq 0$, we can create any ratios $\alpha/\beta, \mu/\nu \in \bar{\mathbb{C}}$. \qed
\end{proof}

\subsection{Solving an inverse eigenvalue problem}\label{sec:Update_SolveIEP}
Updating procedures are also suitable to solve an entire inverse eigenvalue problem.
Start from the trivial solution to a one dimensional IEP, for $z_1\in\mathbb{C}$ and $Q,B,C\in\mathbb{C}^{1\times 1}$,
\begin{equation*}
z_1 Q C = Q B.
\end{equation*}
Clearly $Q=1, C=1$ and $B=z_1$ is a possible solution and updating this solution $m-1$ times to include all nodes and weights will provide a solution to the IEP, tridiagonal or Hessenberg pencil depending on the imposed structure.\\
The implementation of the updating procedures determines their performance. 
If implemented correctly, the updating procedure for the Hessenberg pencil is numerically stable.
The three parts composing updating procedures are enforcing the orthogonality condition, imposing the necessary structure of the pencil and introducing the new pole.
We discuss their numerical stability for the updating procedure for Hessenberg pencils.
Enforcing the orthogonality condition and introducing the new pole corresponds to multiplication with a plane rotation, this can be done numerically stable \cite{Wi63}.
Showing the stability of the procedure that imposes Hessenberg structure is more subtle.
In the proof of Lemma~\ref{lemma:Hess_pencil} it is apparent that a choice must be made whether to use $H$ or $K$ to compute the parameters of the plane rotation which will eliminate the elements on the last row. 
The criterion for this choice which leads to the numerically most stable implementation is, using the notation of the proof of Lemma~\ref{lemma:Hess_pencil},
\begin{equation*}
\text{compute } P_{i+1} \text{ from } 
\begin{cases}
\prod_{l=i}^{1} P_l \widehat{H} \prod_{k=1}^{i}\dot{P}_k, \quad \text{if } \frac{\eta}{\epsilon} < \frac{\delta}{\alpha}\\
\prod_{l=i}^{1} P_l \widehat{K} \prod_{k=1}^{i}\dot{P}_k, \quad \text{else} 
\end{cases}.
\end{equation*}
The elements in the last rows of the pencil, which $P_{i+1}$ should eliminate, are set explicitly to zero and this is numerically stable for the above criterion \cite{CaMaVaWa20}.\\
Since the annihilated elements in the last rows are set explicitly to zero, the resulting structure is exactly a Hessenberg pencil and therefore corresponds exactly to a structured matrix pencil containing recurrence coefficients of rational functions with prescribed poles.
From numerical experiments, however, it will become clear that the poles appearing in the pencil do suffer some accuracy loss.
More details are provided in Section~\ref{sec:numerics}.\\
The numerical analysis of the updating procedure for the tridiagonal pencil IEP is not studied here.
Biorthogonal methods require more components, a manner to deal with breakdowns and numerical breakdowns, great care in computing the eliminators and scalings to balance the order of magnitude between the matrices in the pencil.
This is outside of the scope of this manuscript.\\
We will present some numerical results for the procedures solving the tridiagonal pencil IEP in Section~\ref{sec:numerics}.
These serve as a proof of concept and show the advantage of working with a short recurrence relation, i.e., a tridiagonal instead of a Hessenberg pencil.
The implementation follows Section \ref{sec:TPIEP_Update} exactly and sets all annihilated elements explicitly to zero.
A criterion, as discussed above, is not used because theoretical background is lacking to make an adequate choice.

\section{Numerical analysis}\label{sec:numerics}
The proposed solution strategies, based on Krylov subspace methods and on the updating procedures, are analyzed numerically.
Consider the diagonal matrix of distinct nodes $Z\in\mathbb{C}^{m\times m}$, a weight vector $v\in \mathbb{C}^m$ and a set of poles $\Xi=\{\xi_i\}_{i=1}^{m-1}$, with $\xi_i\in {\mathbb{C}}\backslash \{0\}$ (the exclusion of $\xi_i = 0$ and $\xi_i = \infty$ is done solely for simplicity of notation).
In case of biorthogonality, an extra weight vector $w\in \mathbb{C}^m$ and set of poles $\varPsi$ is provided.
Both solution strategies compute a solution in the form of a pencil $(B,C)$ such that
\begin{equation*}
W^H Z V C = B,
\end{equation*}
with $W^H V = I$, $We_1 = \alpha w$, $Ve_1 = \beta v$ for some constants $\alpha,\beta$ and $(B,C)$ adhering to either Hessenberg or tridiagonal structure.
When computing this in finite precision some errors will arise and these are measured.
The biorthogonality of the formed bases $V,W$ is measured by 
\begin{equation*}
\textrm{err}_{\textrm{o}} := \Vert W^H V - I\Vert_2,
\end{equation*}
where $I$ is a unit matrix of appropriate size and if we consider an orthogonal basis $Q$, then $V=W=Q$.
The accuracy of the recurrence relation, consisting of recurrence matrices $(B,C)$ and basis $V$, is measured by 
\begin{equation*}
\textrm{err}_{\textrm{r}} := \frac{\Vert Z V C - V B\Vert _2}{\max\left(\Vert ZVC \Vert_2, \Vert VB \Vert_2 \right)}.
\end{equation*}
The elements $(B,C)$ represent recurrence coefficients of sequences of biorthogonal rational functions, $\{r_i\}_{i=0}^{m-1}$ and $\{s_i\}_{i=0}^{m-1}$, or a single sequence of orthogonal rational functions $\{r_i\}_{i=0}^{m-1}$. 
The orthogonality of these functions is checked by constructing their moment matrix, which should equal the unit matrix. 
We get for the orthogonal case with inner product $(.,.)$:
\begin{equation*}
\textrm{err}_{\textrm{f}} :=\left\Vert \begin{bmatrix}
(r_i,r_j)
\end{bmatrix}_{i,j=0}^{m-1} - I\right\Vert_2
\end{equation*}
and for the biorthogonal case with bilinear form $\langle.,.\rangle$:
\begin{equation*}
\textrm{err}_{\textrm{f}} :=\left\Vert \begin{bmatrix}
\langle r_i,s_j\rangle
\end{bmatrix}_{i,j=0}^{m-1} - I\right\Vert_2.
\end{equation*}	
The evaluation of the rational functions $\{r_i\}$ and $\{s_i\}$ is done by solving the system of equations obtained by truncating the last column of the matrix on the left-hand side and the last element of the vector on the right-hand side of
\begin{align*}
&\begin{bmatrix}
r_0(z) & r_1(z) & \cdots & r_{n-1}(z)
\end{bmatrix} \begin{bmatrix}
1  \\
0  \\
\vdots & & B-zC & & \\
0  \\
0  
\end{bmatrix} \\
& = \begin{bmatrix}
r_0(z) & 0 & \cdots & 0 & (zB_{n+1,n} - C_{n+1,n}) r_n(z)
\end{bmatrix},
\end{align*}
where $r_0(z)$ is known. 
The condition number of the resulting system, denoted by $\kappa(B,C)$, will be important for the numerical analysis.\\
The final error metric quantifies the accuracy of the poles. The poles of the computed pencil $(B,C)$ are compared to the given poles $\xi_i$,
\begin{equation*}
\textrm{err}_{\textrm{p}} = \max_{1\leq i\leq m-1}
\left\{\frac{\left\vert \frac{B(i+1,i)}{C(i+1,i)} - \xi_i\right\vert}{\vert \xi_i \vert}\right\}.
\end{equation*}
For the TPIEP, also the superdiagonal ratios reveal poles and must be taken into account, which is taken to be the maximum of the above metric and the following
\begin{equation*}
\textrm{err}_{\textrm{p}} = \max_{2\leq i\leq m-1}\left\{\frac{\left\vert \frac{B(i,i+1)}{C(i,i+1)} - \bar{\psi}_{i-1}\right\vert}{\vert \psi_{i-1} \vert}\right\}.
\end{equation*}
Throughout this section all weights are chosen to equal the value 1.
The numerical performance of the solution strategies for the HPIEP is analyzed in Section \ref{sec:HP_numerics} and for the TPIEP in Section \ref{sec:TP_numerics}.\\
Throughout the following discussion, it is important to be aware of the essential difference between both solution procedures, the updating procedure starts from an already known solution to construct the next solution, whereas the Krylov procedure must start over every time the problem changes.
Therefore, the updating procedure is much more efficient in situations where the solution to a related problem is available.
On the other hand, the Krylov procedure possesses all information about the whole problem, which typically leads to a more accurate solution.

\subsection{Hessenberg pencil}\label{sec:HP_numerics}
Two experiments are discussed. The first uses equidistant nodes on the unit circle and highlights the numerical stability of the proposed updating procedure. The second illustrates the influence of the given nodes on the accuracy of the numerical solution by choosing two nodes close to each other.\\
The first experiment uses equidistant nodes on the unit circle, since updating always adds one node and keeps all others fixed, we cannot have equidistant nodes at each step. 
The node is then added at the largest distance from all nodes already generated, exactly in the middle of two adjacent nodes.
This order of adding nodes is chosen because it is a good order for the updating procedure, the order has little effect on the final solution (for the same nodes) but strongly influences the intermediate behavior. \\
The poles $\Xi$ are chosen equidistant on a circle of radius 1.5.
The result for problem sizes $m=3,18,\dots,393$ is shown in Figure \ref{fig:caseUnitaryEquidistant}.
The three metrics for the matrix solution, $\textrm{err}_{\textrm{o}}$, $\textrm{err}_{\textrm{r}}$ and $\textrm{err}_{\textrm{p}}$, show very good accuracy for both procedures, with the Krylov procedure performing slightly better, which can be attributed to the benefit of solving the complete problem every time.
The metric for the orthogonality of the rational functions, $\textrm{err}_{\textrm{f}}$, shows that the updating procedure performs much better than the Krylov procedure.
\begin{figure}[!ht]
	\includegraphics{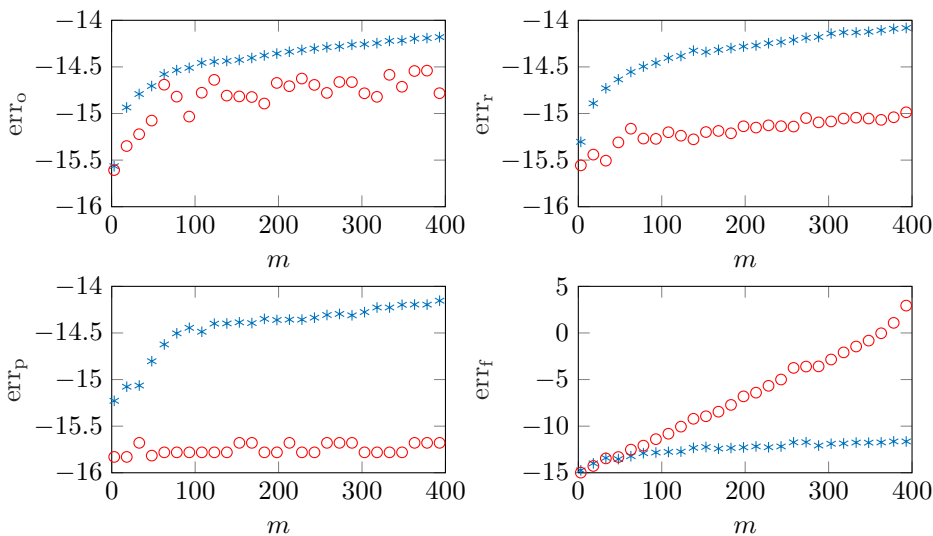}
	\caption{HPIEP with nodes on the unit circle and poles on a circle with radius 1.5. Error metrics for Krylov '$\circ$' and updating '*' procedure in $\log$ scale for problem size $m$.}
	\label{fig:caseUnitaryEquidistant}
\end{figure}

To explain this, we must look at the condition number $\kappa(B,C)$ for the pencil $(B,C)$ obtained by both procedures, shown in Table \ref{table:condition_HPIEP_unitary}.
This table shows that the condition of the system of equations is much larger for the pencil obtained by the Krylov procedure.
The updating procedure performs unitary similarity transformations, therefore if nodes are located on the unit circle then the pencil consists of unitary matrices and this leads to much better conditioning than the Krylov procedure, which does not generate unitary matrices. 
Note that the pencil $(B,C)$ as a whole is unitary in both cases.
\begin{table}[!ht]
	\centering
	\begin{tabular}{l|lllll}
		$m$    & 10  & 100  & 200 & 300 & 400  \\ \cline{1-6} 
		Update & 1.9e01 & 2.3e02  & 4.8e02  & 1.4e03  & 9.1e03\\
		Krylov & 2.9e01 & 2.8e05& 4.0e09 & 3.3e13 & 7.7e17
	\end{tabular}
	\caption{Condition number $\kappa(B,C)$ for the pencil $(B,C)$ obtained by the updating and Krylov procedure for the first experiment, with poles on a circle of radius 1.5, for problem size $m$.}
	\label{table:condition_HPIEP_unitary}
\end{table}
The condition number $\kappa(B,C)$ of the solution obtained by the Krylov procedure depends on the choice of poles, we repeat the above experiment with poles on a circle with radius 3. 
Table \ref{table:condition_HPIEP_unitary_2} shows a much smaller $\kappa(B,C)$ for the Krylov solution.
This illustrates the dependence of $\kappa(B,C)$ of the Krylov solution on the prescribed poles.
The updating solution is much less influenced by the choice of poles.
\begin{table}[!ht]
	\centering
	\begin{tabular}{l|lllll}
		$m$    & 10  & 100  & 200 & 300 & 400  \\ \cline{1-6} 
		Update & 2.0e01 & 2.2e02  & 4.4e02  & 1.4e03  & 9.8e02\\
		Krylov & 2.4e01 & 3.2e02& 6.1e02 & 1.8e03 & 1.4e03
	\end{tabular}
	\caption{Condition number $\kappa(B,C)$ for the pencil $(B,C)$ obtained by the updating and Krylov procedure for the first experiment, with poles on a circle of radius 3, and problem size $m$.}
	\label{table:condition_HPIEP_unitary_2}
\end{table}	

The second experiment shows how the numerical solution of a HPIEP depends on the location of the given nodes.
The nodes are chosen as above, up to the following change, the $m_p$th node is chosen on the circle and close to the $(m_p-1)$th node.
That is, for $m\geq m_p$ nodes, we have $m-1$ equidistant on the circle as above, and a node close to one of these nodes, the distance between these two nodes is given by the angle $\theta$.
For small $\theta$ this leads to an underlying discrete inner product with $m$ nodes that is very close to an inner product with $m-1$ nodes.
Therefore, the $m$th orthogonal rational function will become closer to numerical linear dependence, as $\theta$ gets smaller, with a deterioration of the orthonormality of the generated rational functions as a consequence.
Figure \ref{fig:caseUnitaryEquidistant_pert} shows the results for $m_p=50$, $\theta = 10^{-6}$ and equidistant poles on a circle of radius 3.
The metrics $\textrm{err}_{\textrm{o}}$ and $\textrm{err}_{\textrm{r}}$ behave nicely, as with the first experiment.
The error on the poles obtained by the updating procedure, $\textrm{err}_{\textrm{p}}$, makes a jump when a value $m>m_p$ is reached and stagnates thereafter.
This jump is caused by the small value obtained for the last element in the Hessenberg pencil, this element is the inner product of $r_{m_p-1}(z)$ with $r_{m_p-2}(z)$ which is small due to the similarity of the inner product with $m_p$ and $m_p-1$ nodes.
And this element is used to introduce the new pole, but due to the difference in order of magnitude (about the size of $\theta$), loss of accuracy (about 5-6 digits) is expected.
As predicted, the orthonormality, measured by $\textrm{err}_{\textrm{f}}$, of the complete set of the $m$ orthogonal rational functions deteriorates for $m>m_p$.
However, if we look at the $m-1$ first orthogonal rational functions, the situation is much better. 
This means that the loss of orthogonality, as expected by the closeness of an inner product of $m$ and $m-1$ nodes, is isolated in the $m$th ORF. 
Hence, the first $m-1$ ORFs are still accurately computed.

\begin{figure}[!ht]
	\includegraphics{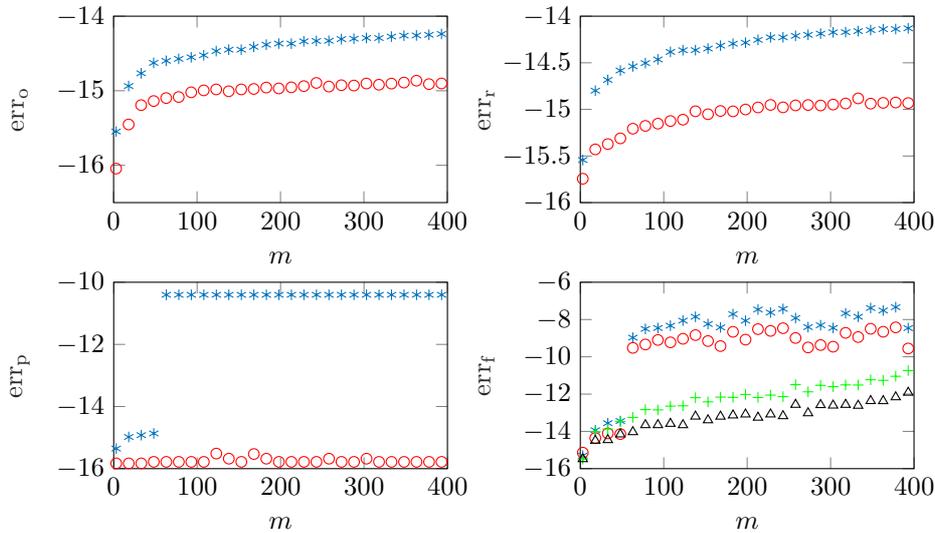}
	\caption{HPIEP with nodes on the unit circle, perturbation of $10^{-6}$ at $m=50$ and poles on a circle with radius 3. Error metrics for Krylov '$\circ$' and updating '*' procedure in $\log$ scale for problem size $m$. The metric $\textrm{err}_\textrm{f}$ for the $m-1$ first ORFs is indicated by '+' for Krylov and '$\triangle$' for updating.}
	\label{fig:caseUnitaryEquidistant_pert}
\end{figure}

\subsection{Tridiagonal pencil}\label{sec:TP_numerics}
The tridiagonal pencil is interesting because of the underlying short recurrence relation. 
If this can be combined with an inner product, which has preferred numerical properties over a general bilinear form, then it will lead to an efficient procedure to generate ORFs in a stable manner.
This scenario occurs for nodes on the real line and poles $\xi_i = \bar{\psi}_i$.
The first experiment chooses Chebyshev nodes, obtained by projecting the nodes of the equidistant unit circle from Section \ref{sec:HP_numerics} onto the real line.
The second experiment serves as a proof of concept, the nodes are chosen equidistant on a thin ellipse, a choice between the unit circle and Chebyshev nodes, where Chebyshev nodes would be the limit case (such an ellipse of height zero).
For both experiments all weights equal the value 1, $v_i=w_i=1$, for all $i$.\\
The metrics for the first experiment, with Chebyshev nodes on the interval $[-1,1]$ and equidistant poles on a circle of radius 3, is provided in Figure \ref{fig:caseHermitianTP}.
As expected from using nonunitary similarity transformations, the procedure is no longer numerically stable, we see a steady deterioration of all metrics.
The errors are however still relatively small, especially for $\textrm{err}_\textrm{f}$. 
Table \ref{table:compare_HP_TP} shows this metric for solutions obtained by the solution procedures for the HPIEP and TPIEP, which are equivalent except for the imposed structure on the pencil.
The Hessenberg pencil achieves only one significant digit more than the tridiagonal pencil, which makes the procedures based on the biorthogonal formulation competitive thanks to the efficiency gained by the underlying short recurrence relation.

\begin{figure}[!ht]
	\includegraphics{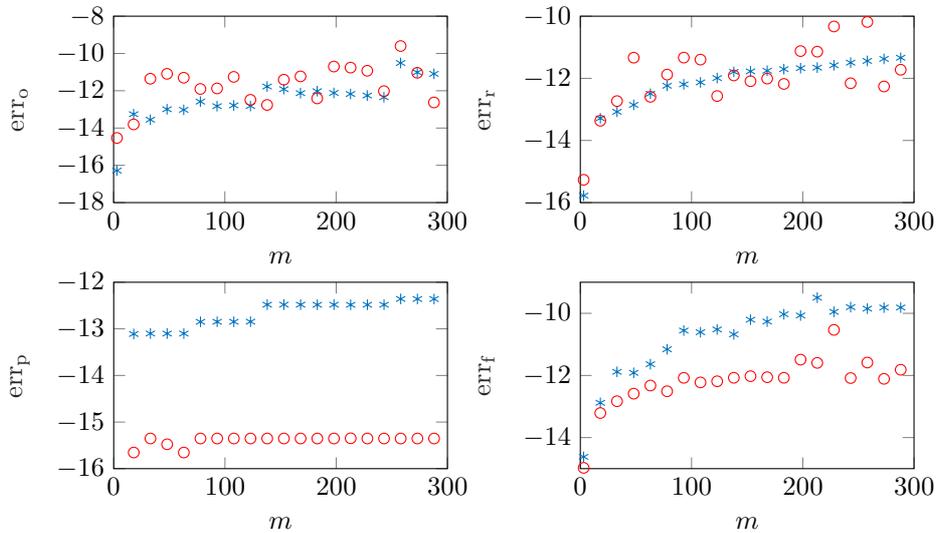}
	\caption{TPIEP with Chebyshev nodes on the interval $[-1,1]$ and equidistant poles on a circle with radius 3. Error metrics for Krylov '$\circ$' and updating '*' procedure in $\log$ scale for problem size $m$.}
	\label{fig:caseHermitianTP}
\end{figure}

\begin{table}[!ht]
	\centering
	\begin{tabular}{l|llll}
		m         & \multicolumn{1}{c}{18} & \multicolumn{1}{c}{93} & \multicolumn{1}{c}{198} & \multicolumn{1}{c}{288} \\ \hline
		Update TP & e-12.9   & e-10.6    &  e-10.1   &  e-9.8   \\
		Krylov TP & e-13.2   &  e-12.1   &  e-11.5   & e-11.8    \\ \hline
		Update HP &  e-13.5  &  e-11.5   &  e-10.6   &  e-10.4   \\
		Krylov HP & e-13.6   &  e-12   &  e-12   &    e-11.9
	\end{tabular}
	\caption{Metric $\textrm{err}_\textrm{f}$ for solution of the IEP with Chebyshev nodes in $[-1,1]$ and poles equidistant on circle with radius 3. Solutions are obtained by solving the HPIEP with updating and Krylov procedure and  by solving the equivalent TPIEP with its respective updating and Krylov procedure.}
	\label{table:compare_HP_TP}
\end{table}

The second experiment uses equidistant nodes on a thin ellipse $x^2+(y/0.01)^2 = 1$, obtained by compressing the unit circle $x^2 + y^2 = 1$ from Section \ref{sec:HP_numerics} in height.
The poles $\Xi$ are chosen equidistant on a circle of radius 3 and $\varPsi$ on a circle of radius 4.
Results of this experiment are shown in Figure \ref{fig:proofOfConcept}.
The updating procedure performs well despite its biorthogonal nature.
The solution obtained by the Krylov procedure deteriorates fast, this can be explained by the fact that it is a general purpose Lanczos-like recurrence relation, whereas the updating procedure is designed for this specific problem.
For poles which are equal, the situation improves, as is shown by the metrics in Figure \ref{fig:proofOfConcept_equalPoles}, where $\Xi=\varPsi$.\\

These results show the potential of the biorthogonal procedures, especially for special cases such as discussed in the first experiment.

\begin{figure}[!ht]
	\includegraphics{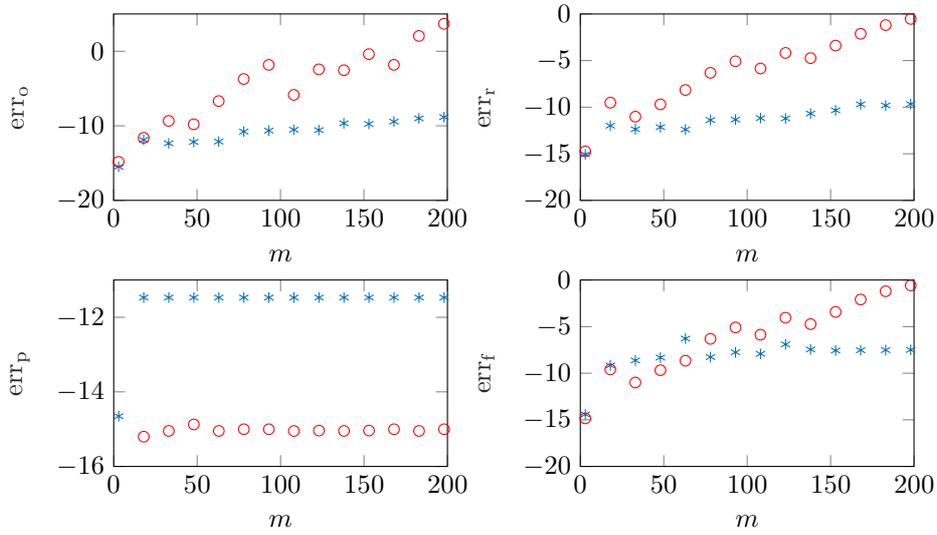}
	\caption{TPIEP with equidistant nodes on an ellipse $x^2 + (y/0.01)^2 = 1$ and poles $\Xi$ and $\varPsi$ on a circle with radius 3 and 4, respectively. Error metrics for Krylov '$\circ$' and updating '*' procedure in $\log$ scale for problem size $m$.}
	\label{fig:proofOfConcept}
\end{figure}


\begin{figure}[!ht]
	\includegraphics{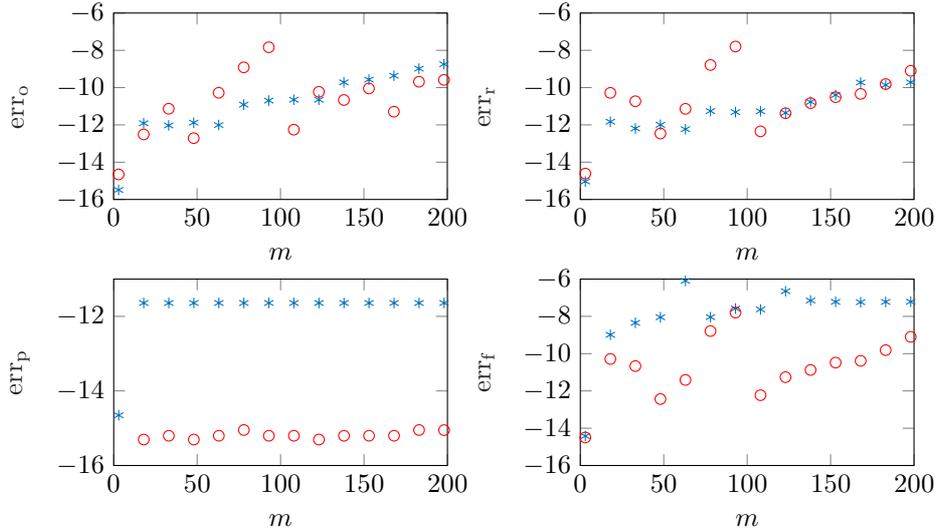}
	\caption{TPIEP with equidistant nodes on an ellipse $x^2 + (y/0.01)^2 = 1$ and poles $\Xi=\varPsi$ on a circle with radius 3. Error metrics for Krylov '$\circ$' and updating '*' procedure in $\log$ scale for problem size $m$.}
	\label{fig:proofOfConcept_equalPoles}
\end{figure}

\section{Conclusion}
The problem of generating sequences of biorthogonal rational functions with prescribed poles is formulated as a problem in numerical linear algebra.
For this so-called inverse eigenvalue problem two solution procedures are proposed, one based on Krylov subspace methods and the other an updating procedure which allows reusing already known solutions.
The updating procedure is numerically stable in case of a sequence of orthogonal rational functions.
The related rational functions retain their orthogonality very well up to modest degree.
The updating procedure has two advantages over the Krylov procedure: by using only unitary similarity transformation it obtains a pencil which consists of better conditioned matrices and it can build further on available solutions without the need to recompute.
Biorthogonal rational functions can be constructed with a short recurrence relation, via an inverse eigenvalue problem for a tridiagonal pencil.
The procedures to compute the recurrence coefficients might, however, suffer from numerical instability.
A trade-off between accuracy and efficiency must therefore be made when choosing either the Hessenberg or tridiagonal pencil formulation.
Subject of future research is to improve the numerical properties of the biorthogonal procedures, since they show potential, especially in special cases such as all nodes located on the real line.

\subsection*{Acknowledgements}
This is a preprint of an article published in Numerical Algorithms. The final authenticated version is available online at: https://doi.org/10.1007/s11075-021-01125-6

\bibliographystyle{siam}
\bibliography{references}  

\end{document}